\documentclass{article}

\usepackage[cp1250]{inputenc}
\usepackage{amsmath}
\usepackage{amssymb}
\usepackage{amsthm}
\usepackage{hyperref}
\numberwithin{equation}{section}
\newcommand\norm[1]{\left\lVert#1\right\rVert}
\usepackage[toc,page]{appendix}
\newtheorem{defi}{Definition}
\newtheorem{lem}[defi]{Lemma}
\newtheorem{rem}[defi]{Remark}

\newtheorem{theorem}[defi]{Theorem}

\newtheorem*{theorem*}{Theorem}
\begin{document}
\title{Particle picture representation of the non-symmetric Rosenblatt process and Hermite processes of any order}
\author{\L ukasz Treszczotko\footnote{Institute of Mathematics, University of Warsaw, Banacha 2 02-097 Warsaw}\\
\href{mailto:lukasz.treszczotko@gmail.com}{lukasz.treszczotko@gmail.com} }

\maketitle

\begin{abstract}
We provide a particle picture representation for the non-symmetric Rosenblatt process and for Hermite processes of any order, extending the result of Bojdecki, Gorostiza and Talarczyk in~\cite{FILT}. We show that these processes can be obtained as limits in the sense of finite-dimensional distributions of certain functionals of a system of particles evolving according to symmetric stable L\'{e}vy motions. In the case of $k$-Hermite processes the corresponding functional involves $k$-intersection local time of symmetric stable L\'{e}vy processes
\end{abstract}
\textup{2000} \textit{Mathematics Subject Classification}: \textup{Primary:60G18} \textup{Secondary:60F17}  

\section{Introduction}

\subsection{Hermite Processes and Generalized Hermite Processes}

	In this paper we study so called Hermite processes and their generalizations. A stochastic process $(X(t))_{t \geq 0}$ is said to be self-similar if there exists a constant $H>0$ such that for any $a>0$ $(X(at))_{t \geq 0}\overset{d}{=}(a^H X(t))_{t \geq 0}$, where the equality is in the sense of finite dimensional distributions. Any self-similar process which also has stationary increments is usually called $H$-sssi. $H$-sssi processes are studied mainly because they are the only possible limits of normalized partial sums of stationary sequences. To be more precise, Lamperti's theorem states that whenever $(X(n))_{n \in \mathbb{Z}}$ is a stationary sequence of random variables and
\begin{equation}\label{lamperti}
\frac{1}{A(N)}\sum_{n=1}^{[Nt]}X(n) \Rightarrow Y(t)
\end{equation}
in the sense of finite dimensional distributions, where $A(N) \rightarrow \infty$ as $N\rightarrow \infty$, then $(Y(t))_{t \geq 0}$ must be an $H$-sssi process. $H$ is usually called the \emph{Hurst coefficient}. It is exactly in this setting that Hermite processes arose in the first place, in the so called \emph{non-central limit theorems} (see~\cite{DM}). Following~\cite{TM3} we briefly sketch it. Let $(\xi_n)_{n\in \mathbb{Z}}$ be a centered stationary Gaussian sequence with variance equal to $1$ such that
\begin{equation}
r(n):=\mathbb{E}(\xi_n \xi_0)=n^{\frac{2H-2}{k}}L(n),
\end{equation}
with $H \in (\frac{1}{2},1)$, $k \geq 1$ and $L$ - a function slowly varying at infinity. Take any function $g:\mathbb{R}\rightarrow \mathbb{R}$ satisfying $\mathbb{E}g(\xi_0)=0$ and $\mathbb{E}g(\xi_0)^2 <\infty$, which has the following expansion in Hermite polynomials:
\begin{equation}
g(x)=\sum_{j=0}^\infty c_j H_j(x),
\end{equation}
where $H_j$ is the $j$-th Hermite polynomial, $c_j=\frac{1}{j!}\mathbb{E}(g(\xi_0)H_j(\xi_0))$ and $k$ is the smallest $j$ with $c_j \neq 0$. If we introduce the following sequence of stochastic processes:
\begin{equation}\label{partial}
Z^{k,n}_H(t):=\frac{1}{n^H}\sum_{l=1}^{[nt]}g(\xi_l),\quad, n\in \mathbb{N},t \geq 0,
\end{equation}
then 
\begin{equation}\label{taqqu_rep}
Z^{k,n}_H \overset{d}{\Rightarrow}c_k Z^k_H,
\end{equation}
where $Z^k_H$ is the k-Hermite process and the convergence holds in the sense of finite-dimensional distributions (from now on we will use the notation $\overset{d}{\Rightarrow}$ to denote this type of convergence). This is how Hermite processes were obtained in the first place and investigating the convergence of partial sums is still a useful way of obtaining new stochastic processes.
\\ 

Hermite processes may also be described with the help of multiple Wiener-It\^{o} integrals. For an introduction to these integrals see~\cite{N}. For $k\in\mathbb{N}$ (using the notation from~\cite{BT}) one can represent a $k$-Hermite process as
\begin{equation}\label{norm_rep}
Z_H^k(t):=a_{k,d}\int_{\mathbb{R}^k}^{'}\int_0^t \prod_{j=1}^k(s-x_j)_+^{d-1}ds W(dx_1)\ldots W(dx_k),
\end{equation}
where $W$ is a two-sided Brownian motion, $\frac{1}{2}(1-\frac{1}{k}) < d < \frac{1}{2}$ is the number satisfying $H=kd-k/2+1$ so that $1/2 <H<1$, $a_{k,d}$ is a positive constant chosen so that $\mathbf{Var}(Z_H^k(1))=1$ and "$'$" above the integral sign indicates that the diagonal is excluded from integration. For our purposes it will be convenient to use a so called \emph{spectral representation} which uses multiple Wiener-It\^{o} integrals as defined in~\cite{MAJOR}:
\begin{equation}\label{spec_rep}
Z_H^k(t)=c_{k,d}\int_{\mathbb{R}^k}^{''}\frac{e^{i(u_1+\ldots +u_k)t}-1}{i(u_1+\ldots +u_k)}|u_1|^{-d}\ldots |u_k|^{-d} \widehat{W}(du_1)\ldots \widehat{W}(du_k);
\end{equation}
here the constant $c_{k,d}$ serves the same purpose as $a_{k,d}$ in~\eqref{norm_rep} and $\widehat{W}$ is the random complex  Gaussian white noise measure on $\mathbb{R}$, where $d$ is given as before. For more about this representation and random spectral measures (which are used extensively in this paper) see Chapter 3 of~\cite{MAJOR}, for a brief overview see Subsection~\ref{fields}.\par
$k$-Hermite processes ``live'' in the $k$-th Wiener chaos. For $k=1$ the $1$-Hermite process is just a fractional Brownian motion and the $2$-Hermite process is called the \emph{Rosenblatt process}. For a comprehensive introduction to the Rosenblatt process see~\cite{RP}. For all $k \in \mathbb{N}$ $k$-Hermite processes have the same covariance given by 
\begin{equation*}
\mathbb{E}(Z_H^k(s)Z_H^k(t))=:R(s,t)=\frac{1}{2}(s^{2H}+t^{2H}-|s-t|^{2H}).
\end{equation*}
\\

It was not at first obvious whether or not Hermite processes were the only self-similar processes with stationary increments in their respective Wiener 
chaoses. The only sssi-process in the first Wiener chaos is the fractional Brownian motion (see Theorem 1.3.3 in~\cite{EM}). This is not true for Wiener chaoses of order $k \geq 2$ and the first example of this fact was the \emph{non-symmetric Rosenblatt process} (see~\cite{TM4} and~\cite{TUDOR} for a nice introduction to this process), which is obtained if we replace the kernel $\prod_{j=1}^k(s-x_j)_+^{d-1}$ in equation~\eqref{norm_rep} (for $k=2$) by 
\begin{equation}
g(x,y)=x^{-1+\alpha/2}y^{-1+\beta/2}\mathbf{1}_{\{x>0,y>0\}},
\end{equation}
with $\alpha, \beta \in (0,1)$ and $\alpha+\beta>1$. Even in this relatively simple case with $k=2$, $\alpha,\alpha',\beta,\beta' \in (0,1)$ and $\alpha+\beta=\alpha'+\beta'>1$ the corresponding non-symmetric processes have different laws for different choices of $\alpha,\alpha',\beta$ and $\beta'$ (see Proposition 3.10 in~\cite{TUDOR}). More generally, the initial kernel can be replaced by even more general functions to obtain the so called \emph{generalized Hermite processes} introduced and investigated in~\cite{BT}.

%In~\cite{BT} generalized Hermite processes were introduced. The authors noticed that the \emph{kernel} $\int_0^t \prod_{j=1}^k(s-x_j)_+^{d-1}ds$ in~\eqref{norm_rep} can be replaced by a more general function chosen so that the process defined by it is still well defined, self-similar with stationary increments. In particular following~\cite{BT} we can consider non-symmetric \emph{kernels} of the form
%\begin{equation}
%g(\mathbf{x})=\prod_{j=1}^k x_j^{\gamma_j}, \quad \mathbf{x}=(x_1,\ldots,x_k)\in\mathbb{R}_+^k,
%\end{equation}
%where $-1<\gamma_j <-1/2$ for $j=1,\ldots,k$ and $-k/2-1/2<\sum_{j=1}^k \gamma_j <-k/2$. In~\cite{BT} the authors also introduced generalized Hermite processes with Hurst parameter $H\in(0,1/2)$ but we are not going to investigate them in this paper. The case when $k=2$ is of particular importance to us. The resulting process is then called the \emph{non-symmetric Rosenblatt process} and the kernel defining it is given by
%\begin{equation}
%g(x,y)=x^{-1+\alpha/2}y^{-1+\beta/2}\mathbf{1}_{\{x>0,y>0\}},
%\end{equation}
%with $\alpha, \beta \in (0,1)$ and $\alpha+\beta>1$. 

\subsection{Representation of Hermite Processes}\label{repr}

Hermite processes arise naturally as limits of normalized partial sums of stationary sequences as in~\eqref{partial}. Recently in~\cite{FILT} a different type of limit theorem was proved. It was shown that the Rosenblatt process can be obtained from a Poisson system of particles evolving according to $\alpha$-stable processes. Our aim is to extend this representation to the general $k$-Hermite processes and the non-symmetric Rosenblatt process. Let us briefly sketch the particle system we are going to use.

Let $(x^j)$ be a Poisson system with Lebesgue intensity measure on $\mathbb{R}$ and let $(\xi^j)_{j=1}^\infty$ be independent symmetric $\alpha$-stable Levy processes with the index of stability $\alpha \in (0,1)$. Notice that we only consider the values of the parameter $\alpha$ for which $(\xi_t)_{t \geq 0}$ is transient. We also assume that these processes are independent of the points $(x^j)$. In the end, assume that $(\sigma_j)_{j=1}^\infty$ are i.i.d random variables such that $\mathbb{P}(\sigma_1=1)=\mathbb{P}(\sigma_1=-1)=\frac{1}{2}$ and that these variables are independent of everything else. The particle system is given by $(x^j+\xi^j_t)_{t \geq 0}$. Thus the initial position of the particles is given by the points $(x^j)$ and they evolve independently according to the symmetric $\alpha$-stable processes. Furthermore, we independently assign charges $\sigma_j$ to these particles.  This is the underlying system which will give rise to the stochastic processes studied in this paper and will be present throughout our work.

In~\cite{FILT} it was shown that the process given by
\begin{equation}\label{talrep}
\xi_t^T =\frac{1}{T}\sum_{j \neq k}\sigma_j \sigma_k  \langle \Lambda(x^j+\xi^j,x^k+\xi^k;T),\mathbf{1}_{[0,t]}\rangle,\quad t\geq 0,
\end{equation}
where $\Lambda$ is the \emph{intersection local time} of two independent $\alpha$-stable L\'{e}vy processes (see section 2.3 in~\cite{FILT},~\cite{AP} and section~\ref{sec_ilt} in in this paper), converges, as $T\rightarrow \infty$, (up to a constant) for $\alpha \in (1/2,1)$, in $C([0,\tau])$ for $\tau \in (0,\infty)$, to the Rosenblatt process with the Hurst coeficient $H=\alpha$. We will show how a non-symmetric Rosenblatt process is obtained from the same particle system. We will also extend the result of~\cite{FILT} to $k$-Hermite processes for $k \geq 3$.

\subsection{Results}
First we will state a limit theorem leading to a non-symmetric Rosenblatt process. Consider the particle system described in Section~\ref{repr} and let $\beta>\alpha$. Define
\begin{equation}\label{ARPfun}
\eta^T_t =\frac{1}{T}\sum_{j \neq k}\sigma_j \sigma_k \int_0^T \int_0^T \mathbf{1}_{[0,t]}(x^j+\xi^j_r)\frac{1}{|x^k+\xi^k_s-x^j-\xi^j_r|^{1-\frac{\beta-\alpha}{2}}}drds
\end{equation}
for $T>0,t\geq 0$. The fact that the above functional is well defined (in the sense that the sum in~\eqref{ARPfun} converges in $L^2(\Omega)$) will be shown in Subsection~\ref{arpprem}. The first of the two main results in this paper is the following theorem.
\begin{theorem}\label{mainARP}
Let $\alpha$ and $\beta$ be such that $1>\beta>\alpha>0$, $\alpha +\beta>1$. Then, as $T\rightarrow \infty$, the processes $(\eta^T_t)_{t \geq 0}$ converge in the sense of finite dimensional distributions to the non-symmetric Rosenblatt process with parameters $(\alpha,\beta)$, up to a multiplicative constant.
\end{theorem}
The quite explicit formulation of the approximating process in Theorem~\ref{mainARP} is what makes it particularly appealing.\\

The second question we set out to answer was whether the representation given by~\eqref{talrep} can be extended to $k$-Hermite processes for general $k\geq 2$. To formulate our result we must first introduce the so called \emph{$k$-intersection local time} ($k$-ILT), which is an extension of the notion of \emph{intersection local time} (ILT) and was first considered in~\cite{SI}. Informally $k$-ILT of cadlag processes $\rho^1,\ldots,\rho^k$ at time $T \geq 0$ (denoted here by $\Lambda^{(k)}$) can be defined by
\begin{multline}\label{kinter}
\left \langle \Lambda^{(k)}(\rho^1,\ldots,\rho^k;T),\phi \right \rangle =\\
=\int_{[0,T]^k}\phi(\rho^1_{s_1})\delta_0(\rho^2_{s_2}-\rho^1_{s_1})\ldots \delta_0(\rho^k_{s_k}-\rho^1_{s_1})ds_1\ldots ds_k,
\end{multline}
where $\delta_0$ is the Dirac distribution at $0$ and $\phi \in \mathcal{S}$ (the Schwartz space of rapidly decreasing function). One gives a meaning to~\eqref{kinter} by approximating $\delta_0$ by smooth fuctions. The precise definition and the proof of existence of ILT in the case of independent symmetric $\alpha$-stable L\'{e}vy processes is given in Section~\ref{kherm}, which is an extension of Proposition 5.1 in~\cite{PP}. The answer to the second question is provided by the following theorem.

%In~\cite{PP} \emph{intersection local time} of two cadlag processes was studied in great detail and some quite general criteria for its existence were proved (see proposition 3.3,~\cite{PP}). For any pair of cadlag processes $\eta^1=(\eta^1_t)_{t \geq 0}$, $\eta^2=(\eta^2_t)_{t \geq 0}$, $f \in \mathcal{F}$, $\epsilon>0$, $\phi \in \mathcal{S}(\mathbb{R})$ we define
%\begin{equation}\label{ILTdefi}
%\langle \Lambda^f_\epsilon(\eta^1,\eta^2;t),\phi \rangle = \int_0^t \int_0^t \phi(\eta^1_u)f_\epsilon(\eta^2_v-\eta^1_u)dudv, \qquad t>0,\phi \in \mathcal{S}(\mathbb{R}^d).
%\end{equation}
%If for any given pair of starting points $x,y \in \mathbb{R}$ and any $\phi \in \mathcal{S}(\mathbb{R}),\quad t >0$ the $L^2(\Omega)$-limit as $\epsilon \rightarrow 0$ of $\langle \Lambda^f_\epsilon(\eta^1,\eta^2;t),\phi \rangle$ exists as is independent of the choice of $f$, then we denote this limit by $\langle \Lambda(\eta^1,\eta^2;t),\phi \rangle$. Following remark 3.2 in~\cite{PP} we can then infer the existence of the $\mathcal{S'}(\mathbb{R})$-valued stochastic process $\left(\Lambda(\eta^1,\eta^2;t)\right)_{t>0}$. We call this process the \emph{intersection local time} of $\eta_1$ and $\eta_2$. 

\begin{theorem}\label{mainK}
Let $k \geq 2$ be a natural number and $\alpha \in (1-\frac{1}{k},1)$. For $T>0$ we denote
\begin{equation}\label{Kfun}
\rho^T_t:=\frac{1}{T^{k/2}}\sum_{j_1 \neq  \ldots \neq j_k}\sigma_{j_1}\ldots \sigma_{j_k}\langle \Lambda^{(k)}(x^{j_1}+\xi^{j_1},\ldots,x^{j_k}+\xi^{j_k};T),\mathbf{1}_{[0,t]} \rangle, t \geq 0.
\end{equation}
Then, as $T \rightarrow \infty$, the process $(\rho^T_t)_{t\geq 0}$ converges in the sense of finite-dimensional distributions to the $k$-Hermite process $Z_H^k$ with the Hurst coefficient equal to $H=1-(1-\alpha)k/2$.
\end{theorem}
The main scheme of the proofs of Theorems~\ref{mainARP} and~\ref{mainK} is similar to the one employed in~\cite{FILT}, in particular the idea of using Wick products of an appropriate $\mathcal{S}'$-valued random variable. To stress some of the main differences and difficulties  that had to be overcome in our case let us point out that in the case of Theorem~\ref{mainARP} it was at first not at all clear what functional of a particle system can be used to approximate the non-symmetric Rosenblatt process. Also, since the functional is different, we need to use different approximations.

In case of Theorem~\ref{mainK} it was more or less clear that one should use~\eqref{Kfun} as the approximating process. However, due to the fact that now we have to deal with $k$-intersection local times and Wick products of order $k$ there are some non-trivial technical difficulties (see the proof of~\eqref{hardpart}).

Additionally, in the proof of the representation of the symmetric Rosenblatt process in~\cite{FILT} the identification of the limiting distribution was done using the cumulants and the fact that the finite-dimendional distributions of this process are determined bu its moments. In our paper we will take a different route and utilize the It\^{o} formula for multiple Wiener-It\^{o} integrals (Theorem 4.3 in~\cite{MAJOR}).

The paper is organized as follows. In Section~\ref{NB} we fix the notation and introduce some of the concepts which will be later used extensively to prove the results. Section~\ref{NSRP} contains proof of Theorem~\ref{mainARP} and in Section~\ref{RHP} we discuss the existence of $k$-intersection local time and prove Theorem~\ref{mainK}.

\section{Notation and Background}\label{NB}

\subsection{Notation}\label{sec_ilt}

Throughout the rest of the paper, by $\mathcal{S}(\mathbb{R}^d)$ we will denote the Schwartz space of \emph{real-valued} smooth rapidly decrasing functions on $\mathbb{R}^d, d\in \mathbb{N}$, $\mathcal{S}'(\mathbb{R}^d)$ will be the space tempered distributions. Let $\mathcal{F}$ denote the class of non-negative symmetric, infinitely differentiable functions on $\mathbb{R}$ with with support in $B(0,1)=\{x\in \mathbb{R}:|x|<1\}$ satisfying $\int_{\mathbb{R}}f(x)dx=1$. These functions will be used to to approximate Dirac delta distributions. For any $f\in \mathcal{F}$, $\epsilon >0$ put
\begin{equation*}
f_\epsilon (x)=\epsilon^{-d}f(\frac{x}{\epsilon}), \qquad x \in \mathbb{R}.
\end{equation*}
We will also use the following definition of the Fourier transform. For $\phi \in \mathcal{S}(\mathbb{R}^d)$ 
\begin{equation*}
\widehat{\phi}(x):=\int_{\mathbb{R}^d}e^{ixz}\phi(z)dz,\quad x \in \mathbb{R}^d.
\end{equation*}
Throughout the paper $\lambda_1$ will denote one-dimensional Lebesgue measure on $\mathbb{R}$ and $\Rightarrow$ will denote convergence in law.
%In~\cite{PP} \emph{intersection local time} of two cadlag processes was studied in great detail and some quite general criteria for its existence were proved (see proposition 3.3,~\cite{PP}). For any pair of cadlag processes $\eta^1=(\eta^1_t)_{t \geq 0}$, $\eta^2=(\eta^2_t)_{t \geq 0}$, $f \in \mathcal{F}$, $\epsilon>0$, $\phi \in \mathcal{S}(\mathbb{R})$ we define
%\begin{equation}\label{ILTdefi}
%\langle \Lambda^f_\epsilon(\eta^1,\eta^2;t),\phi \rangle = \int_0^t \int_0^t \phi(\eta^1_u)f_\epsilon(\eta^2_v-\eta^1_u)dudv, \qquad t>0,\phi \in \mathcal{S}(\mathbb{R}^d).
%\end{equation}
%If for any given pair of starting points $x,y \in \mathbb{R}$ and any $\phi \in \mathcal{S}(\mathbb{R}),\quad t >0$ the $L^2(\Omega)$-limit as $\epsilon \rightarrow 0$ of $\langle \Lambda^f_\epsilon(\eta^1,\eta^2;t),\phi \rangle$ exists as is independent of the choice of $f$, then we denote this limit by $\langle \Lambda(\eta^1,\eta^2;t),\phi \rangle$. Following remark 3.2 in~\cite{PP} we can then infer the existence of the $\mathcal{S'}(\mathbb{R})$-valued stochastic process $\left(\Lambda(\eta^1,\eta^2;t)\right)_{t>0}$. We call this process the \emph{intersection local time} of $\eta_1$ and $\eta_2$. 
%\\

Since we are interested in the convergence of stochastic processes in the sense of finite dimensional distributions it is convenient to introduce the following class of functions. Let $\mathcal{A}$ be the family of functions of the form
\begin{equation}\label{fdd}
\psi=\sum_{j=1}^m a_j\mathbf{1}_{I_j},
\end{equation}
where $a_j \in \mathbb{R}$ and $I_j$ is a bounded interval for each $j=1,\ldots,m$. For $g \in \mathcal{F}$ and $\psi \in \mathcal{A}$ we will write $\psi_\kappa:=\psi \ast g_\kappa$, without explicitly referring to the function $g$ to make the the notation more transparent. As it happens, we will always require that in the limit the particular choice of $g$ is irrelevant as far as our purposes are concerned. Notice that
\begin{equation}
\widehat{\psi_\kappa}(x)=\widehat{\psi}(x)\widehat{g_\kappa}(x)=\widehat{\psi}(x)\widehat{g}(\kappa x),
\end{equation}
so that
\begin{equation}
|\widehat{\psi_\kappa}(x)|\leq |\widehat{\psi}(x)|,\quad x\in \mathbb{R},
\end{equation}
since $|\widehat{g}(z)|\leq 1$ for all $z\in \mathbb{R}$.

\subsection{Generalized Gaussian Random Fields and Fractional Brownian Motion}\label{fields}
One of the important tools which we will be using in this paper are $\mathcal{S'}$-valued random variables. In particular, we will be working with centered Gaussian  $\mathcal{S'}$-valued random variables . For each $\alpha <1$ there exists a centered Gaussian $\mathcal{S'}$-random variable $X$ with covariance functional given by
\begin{equation}
\mathbb{E} \langle X,\phi \rangle \langle X,\psi \rangle=\frac{1}{\pi}\int_\mathbb{R}\widehat{\phi}(x)\overline{\widehat{\psi}(x)}|x|^{-\alpha}dx,\quad \phi,\psi \in \mathcal{S}(\mathbb{R}),
\end{equation}
The spectral measure of this field is given by $G(dx)=|x|^{-\alpha}dx$. As noted in~\cite{FILT}, for $\alpha \in (0,1)$,  $X$ can be approximated by the normalized total charge occupation of our particle system from Subsection~\ref{repr} on the interval $[0,T]$, that is by a functional given by
\begin{equation}\label{x_t}
\langle X_T,\phi \rangle=\frac{1}{\sqrt{T}}\sum_j \sigma_j \int_0^T\phi(x_j+\xi_s^j)ds,\quad \phi \in \mathcal{S}(\mathbb{R}). 
\end{equation}
By an $L^2$ extension we may evaluate $X$ on functions from a much wider class than $\mathcal{S}(\mathbb{R})$ and in fact $(\langle X,\mathbf{1}_{[0,t]
} \rangle )_{t \geq 0}$ is up to a constant the fractional Brownian motion with Hurst coeficient equal to $H=\frac{1+\alpha}{2}$. The particle system we are working with can be used to approximate this process as in the theorem below.
\begin{theorem}[Theorem 2.1 in~\cite{FILT}]\label{FBR}
For $\alpha \in (0,1)$, as $T \rightarrow \infty$, we have:
\begin{itemize}
\item[(i)]
$X_T \Rightarrow X$ in $\mathcal{S'}(\mathbb{R})$,
\item[(ii)]
$(\langle X_T,\mathbf{1}_{[0,t]} \rangle)_{t \geq 0}$ converges in the sense of finite dimensional distributions to $(KB_t^H)_{t \geq 0}$ where $K$ is a constant and $B^H$ is a fractional Brownian motion with Hurst coefficient $H=\frac{1+\alpha}{2}$.
\end{itemize}
\end{theorem}

\begin{rem}\label{spectralm}
The random field $(\langle X,\phi \rangle )_{\phi \in \mathcal{S}(\mathbb{R})}$ can be used (see chapter 3 in~\cite{MAJOR} for details) to construct a random spectral measure $Z_G$ associated with this field such that $\langle X,\phi \rangle=\int_{\mathbb{R}}\widehat{\phi}(x)Z_G(dx)$ for $\phi \in \mathcal{S}$. We will use it extensively throughout this paper.
\end{rem}

\section{Non-symmetric Rosenblatt Process}\label{NSRP}

\subsection{Preliminaries}\label{arpprem}

Let us denote 
\begin{equation}\label{t1}
\langle \Delta(x+\xi^1,y+\xi^2;T),\phi \rangle=\int_0^T \int_0^T\phi(x+\xi^1_r)|y+\xi^2_s-x-\xi^1_r|^{\frac{\beta-\alpha}{2}-1}drds,
\end{equation}
where $\xi^1,\xi^2$ are independent symmetric $\alpha$-stable L\'{e}vy processes. According to Lemma~\ref{khermiteperm} in the Appendix to show that~\eqref{t1} is well defined (in the sense that the series converges in $L^2(\Omega)$) it suffices to show that for any $t >0$ and independent symmetric $\alpha$-stable processes $\xi^1$, $\xi^2$  we have 
\begin{equation}\label{t2}
\langle \Delta(\cdot+\xi^1,\cdot+\xi^2;T),\mathbf{1}_{[0,t]} \rangle \in L^2(\mathbb{R}\times \mathbb{R} \times \Omega,\lambda_1 \otimes \lambda_1 \otimes \mathbb{P}),
\end{equation}
which is done in the following lemma.
\begin{lem}\label{dod}
Let $\alpha,\beta \in (0,1)$ be such that $\alpha+\beta>1$, $\beta>\alpha$ and let $\phi\in L^1(\mathbb{R})$ be bounded. Then
\begin{equation}
I:=\mathbb{E}\left(\int_{\mathbb{R}^2}\left(\int_0^T\int_0^T\phi(x+\xi^1_r)|y+\xi^2_s-x-\xi^1_r|^{\gamma-1}dsdr\right)^2dxdy\right)\leq CT^2<\infty,
\end{equation}
where $\gamma:=\frac{\beta-\alpha}{2}$ and $C$ is a constant.
\end{lem}
\begin{proof}
Without loss of generality we may assume that $\phi\geq 0$. We have 
\begin{eqnarray}
I&=&\mathbb{E}\int_{\mathbb{R}^2}\int_{[0,T]^4}\phi(x+\xi^1_r)|y+\xi^2_s-x-\xi^1_r|^{\gamma-1}\nonumber\\
&& \:\phi(x+\xi^1_u)|y+\xi^2_v-x-\xi^1_u|^{\gamma-1}drdsdudv dxdy\\
&=&\mathbb{E}\int_{\mathbb{R}^2}\int_{[0,T]^4}\phi(x)|y-x|^{\gamma-1}\phi(x+\xi^1_u-\xi^1_r)\nonumber\\
&& \:|y+\xi^2_v-\xi^2_s-x-(\xi^1_u-\xi^1_r)|^{\gamma-1}drdsdudv dxdy\\
&=&\int_{\mathbb{R}^4}\int_{[0,T]^4}\phi(x)|y-x|^{\gamma-1}\phi(z)|w-z|^{\gamma-1}\nonumber\\
&& \:\phi(z)|w-z|^{\gamma-1}p_{|u-r|}(z-x)p_{|v-s|}(w-y)drdudsdvdxdydzdw,
\end{eqnarray}
where $p$ is the $\alpha$-stable transition density. The second equality uses Fubini theorem. Since for $x \neq 0$ $\int_0^\infty p_s(x)ds=C|x|^{\alpha-1}$ for some constant $C=C(\alpha)$ and 
\begin{equation}
\int_0^T \int_0^T p_{|u-r|}(x)drdu \leq \frac{2CT}{|x|^{1-\alpha}},
\end{equation}
we see that $I$ can be bounded by
\begin{equation}
C^2T^24\int_{\mathbb{R}^2}\phi(x)|y-x|^{\gamma-1}\phi(z)|w-z|^{\gamma-1}\frac{1}{|z-x|^{1-\alpha}}\frac{1}{|w-y|^{1-\alpha}}dxdydzw.
\end{equation}
Notice that
\begin{equation}
\int_\mathbb{R}|y-x|^{\gamma-1}|w-y|^{\alpha-1}dy=C'(\alpha,\beta)|w-x|^{\frac{\alpha+\beta}{2}-1},
\end{equation}
with $C'(\alpha,\beta)$ being a constant. Similarly
\begin{equation}
\int_\mathbb{R}|w-x|^{\frac{\alpha+\beta}{2}-1}|w-z|^{\gamma-1}dw=C''(\alpha,\beta)|z-x|^{\beta-1}.
\end{equation}
Thus
\begin{eqnarray}\label{crimea}
I&\leq&C'''(\alpha,\beta)T^2\int_{\mathbb{R}^2}\phi(x)\phi(z)|z-x|^{\alpha+\beta-2}.
\end{eqnarray}
The right-hand side of~\eqref{crimea} is finite since $\phi$ is bounded and in $L^1(\mathbb{R})$ and $\alpha+\beta>1$.
\end{proof}

Let us briefly discuss the main ideas behind the proof of convergence of finite-dimensional distributions of $\eta^T$ to those of the Rosenblatt process. We will show that the functional $\eta^T_t$ is close to $\langle :X_T \otimes X_T:,\Phi_t \rangle$, where $:X_T \otimes X_T:$ is the Wick product of the process $X_T$ defined by~\eqref{x_t} and $\Phi_t(x,y)=\mathbf{1}_{[0,t]}(x)|x-y|^{\frac{\beta-\alpha}{2}-1}$, $x,y\in \mathbb{R}$. Recall that the Wick product $:X_T \otimes X_T:$ is defined in the following way. First, for $\Phi \in \mathcal{S}(\mathbb{R}^2)$ of the form
\begin{equation}\label{t3}
\Phi=\sum_{j=1}^m \phi_j\otimes \psi_j,
\end{equation}
 where $\phi_j,\psi_j \in \mathcal{S}$, for $j=1,\ldots,m$, we set 
\begin{equation}\label{t4}
\langle :X_T\otimes X_T:,\Phi \rangle:= \sum_{j=1}^m\Big( \langle X_T,\phi_j \rangle \langle X_T,\psi_j \rangle - \mathbb{E}(\langle X_T,\phi_j \rangle \langle X_T,\psi_j \rangle)\Big).
\end{equation}
~\eqref{t4} can then be extended to arbitrary $\Phi \in \mathcal{S}(\mathbb{R}^2)$. Next we would like to use Theorem~\ref{FBR} to obtain that $\langle :X_T\otimes X_T:,\Phi_t \rangle$ converges, as $T \Rightarrow \infty$, in distribution to $\langle :X\otimes X:,\Phi_t \rangle$. Finally we will show that $\left(\langle :X\otimes X:,\Phi_t \rangle\right)_{t \geq 0}$ is up to a constant a non-symmetric Rosenblatt process. One of the difficulties lies in the fact that $\Phi_t$ is not in $\mathcal{S}(\mathbb{R}^2)$ and we must approximate it by functions from the Schwartz space. $\langle :X\otimes X:,\Phi_t \rangle$ is then understood as a limit under these approximations.
%The functional $\eta_t^T$ is closely related to $X_T$ defined in~\eqref{x_t}. In fact we show that $\eta_t^T$ is close to $\langle :X_T\otimes X_T:\Phi_t \rangle$, where  $:X_T\otimes X_T:$ is the Wick product of $X_T$ and $X_T$, i.e. for functions $\Phi$ of the form $\Phi=\sum_{j=1}^m \phi_j\otimes \psi_j$, where $\phi_j,\psi_j \in \mathcal{S}$ for $j=1,\ldots,m$, we define
%\begin{equation}
%\langle :X_T\otimes X_T:,\Phi \rangle:= \sum_{j=1}^m\Big( \langle X_T,\phi_j \rangle \langle X_T,\psi_j \rangle - \mathbb{E}(\langle X_T,\phi_j \rangle \langle X_T,\psi_j \rangle)\Big),
%\end{equation}
%and $\Phi_t(x,y)=\mathbf{1}_{[0,t]}(x)|x-y|^{\frac{\beta-\alpha}{2}-1}$, $x,y\in \mathbb{R}$. Notice that using Lemma 6.3 from~\cite{FILT} the definition of $\langle :X_T\otimes X_T:,\Phi \rangle$ can be extended to any $\Phi \in \mathcal{S}(\mathbb{R}^2)$. To prove our result we would want to use Theorem~\ref{FBR}. The difficulty, however, lies in the fact that $\Phi_t$ is not in $\mathcal{S}(\mathbb{R}^2)$ and we must approximate it by functions from the Schwartz space. 

We now proceed to discussing our approximating functions and some of their properties. For convenience let us set $\gamma=\frac{\beta-\alpha}{2}$. We are going to approximate the function $y\mapsto |y|^{\gamma-1}$ by the convolution $\int_\mathbb{R}|y-z|^{\gamma-1}f_\epsilon(z)dz$ (where $f\in \mathcal{F}$ and $\epsilon>0$), and then use the fact that as $\epsilon \rightarrow 0$ the integral converges (up to a constant depending only on $\gamma$) to $|y|^{-\gamma}$. However, the function $y \mapsto \int_\mathbb{R}|y-z|^{\gamma-1}f_\epsilon(z)dz$ still does not belong to $\mathcal{S}(\mathbb{R})$ as it vanishes slowly. To overcome this obstacle 
%Note that $\gamma \in (0,\frac{1}{2})$ and the function $y \mapsto |y|^{\gamma-1}$ is related to the $0$-potential of the $\gamma$-stable semigroup. On the one hand we have 
%\begin{eqnarray}
%(U\phi)(x) &:=& \int_0^\infty (\mathcal{T}_t \phi)(x)dt\nonumber \\
%&=& \int_0^\infty \left(\mathbb{E}(\phi(x+\rho_t))\right)dt\nonumber \\
%&=& \int_0^\infty \left( \frac{1}{2\pi}\int_\mathbb{R} e^{-izx}\hat{\phi}(z)e^{-t|z|^\gamma}dz \right)dt\nonumber\\
%&=& \frac{1}{2\pi}\int_\mathbb{R}e^{-izx}\hat{\phi}(z)|z|^{-\gamma}dz,x\in \mathbb{R},
%\end{eqnarray}
%where $\phi \in \mathcal{S}(\mathbb{R})$ and $(\rho_t)_{t\geq 0}$ is a symmetric $\gamma$-stable L'{e}vy motion. On the other hand
%\begin{equation}
%U\phi(x)=c_\gamma \int_\mathbb{R}\frac{1}{|x-y|^{1-\gamma}}\phi(y)dy,\quad x \in \mathbb{R},
%\end{equation}
%for some constant $c_\gamma$ depending only on $\gamma$. 
take $\delta \in (0,1)$, let $h_\delta(x):=|x|^{\gamma-1}\mathbf{1}_{\delta < |x| < \frac{1}{\delta}}$ and put 
\begin{equation}\label{vdef}
V^\delta \phi(x):=\int_\mathbb{R}h_\delta(x-y)\phi(y)dy,\quad \phi \in \mathcal{S},x \in \mathbb{R}.
\end{equation}
We approximate the function $y\mapsto |y|^{\gamma -1}$ by $V^\delta f_\epsilon$. Let us also define $V\phi(x):=\lim_{\delta \rightarrow 0_+}V^\delta \phi(x)$. This limit exists as long as $\int_\mathbb{R} |x-y|^{\gamma-1}|\phi(y)|dy < \infty$. It is easy to see that for $f \in C^\infty_c(\mathbb{R})$, $V^\delta f \in C^\infty_c(\mathbb{R})$ for any $\delta \in (0,1)$. The Fourier transform of $V^\delta \phi$ is given by
\begin{equation}\label{vtrans}
\widehat{V^\delta \phi}(x)=\widehat{h_\delta}(x)\widehat{\phi}(x),\quad x \in \mathbb{R},
\end{equation}
for $\phi \in \mathcal{S}$.
In the sequel we will need several of properties of operators $V^\delta$. We list them in the lemma below.
\begin{lem}\label{vprop}
Let $f$ be in $\mathcal{F}$. The operators $V^\delta$ defined by~\ref{vdef} have the following properties:
\begin{itemize}
\item[(i)] for each $\delta \in (0,1)$ $V^\delta f$ is nondecreasing as $\delta \searrow 0$ and 
\begin{equation}\label{v2def}
Vf(x)=\lim_{\delta \rightarrow 0_+}V^\delta f(x);
\end{equation}
\item[(ii)]
\begin{equation}
\lim_{\epsilon \rightarrow 0_+} Vf_\epsilon(x)=|x|^{\gamma-1},\quad x \neq 0;
\end{equation}
%and by Jennsen inequality
%\begin{equation}
%Vf_\epsilon(x)\geq |x|^{\gamma -1}\mathbf{1}_{\{|x|>\epsilon \}},
%\end{equation}
\item[(iii)]
for each $\epsilon>0$ 
\begin{equation}
Vf_\epsilon(x)\leq \norm{f}_\infty \frac{2^{2-\gamma}}{\gamma}|x|^{\gamma-1}.
\end{equation}
\end{itemize}
\end{lem}
\begin{proof}
Parts $(i)$ and $(ii)$ are obvious. 
%Take $x$ such that $|x|>\epsilon$ and let $Y$ by a random variable with density given by $f_\epsilon$. Since the function $z \mapsto |z|^{\gamma -1}$ is convex on $(0,\infty)$, we have by Jensen inequality
%\begin{eqnarray}
%Vf_\epsilon(x)=\mathbb{E}\left(\frac{1}{|x-Y|^{1-\gamma}}\right)\geq |x|^{\gamma-1}.
%\end{eqnarray}
%This proves part $(ii)$. 
To prove part $(iii)$ fix $\epsilon>0$ and consider two cases. Assume first that $|x|\geq 2\epsilon$. Then $Vf_\epsilon(x) \leq 1/|x-\epsilon|^{1-\gamma} \leq 2^{1-\gamma}|x|^{\gamma-1}$. If $|x| < 2\epsilon$ then 
\begin{equation}
Vf_\epsilon(x)=\int_{-\epsilon}^\epsilon \frac{1}{|x-y|^{1-\gamma}}\frac{1}{\epsilon}f\left(\frac{y}{\epsilon}\right)dy \leq \norm{f}_\infty \frac{1}{\epsilon}\int_{-\epsilon}^\epsilon \frac{1}{|y|^{1-\gamma}}dy=\norm{f}_\infty \frac{2^{2-\gamma}}{\gamma}|x|^{\gamma-1},
\end{equation}
and this finishes the proof.
\end{proof}

%\begin{equation*}
%h_{\delta'}:=
%\begin{cases}
%|z|^{\gamma-1} & \text{if } \delta'<|z| < \frac{1}{\delta'},\\
%0 & \text{if } |z|<\delta',\\
%0 & \text{if } |z| > \frac{1}{\delta'},
%\end{cases}
%\end{equation*}
%and put $g_{\delta,\delta'}(z):=(f_{\delta'}*h_\delta)(z)$ for some $f \in \mathcal{F}$. The choice of $f$ is irrelevant as far as our intentions are concerned (as can be easiliy seen from the calculations) and we will not refer to it in the future. For notational convenience we will always take $\delta=\delta'$. Then, define 
%\begin{equation}\label{udef}
%U^{\delta}\phi(x):=\int_\mathbb{R}g_\delta(x-y)\phi(y),\quad x\in \mathbb{R}
%G^{\delta}\phi(x):=c_\gamma\int_\mathbb{R}g_\delta(x-y)\phi(y)dy,\quad x\in \mathbb{R},\phi \in \mathcal{S}(\mathbb{R}),
%\end{equation}
%and note that $U^{\delta}\phi$ is in $\mathcal{S}(\mathbb{R})$ and
%\begin{equation}\label{point1}
%\widehat{U^{\delta}\phi}(x)=\widehat{g_\delta}(x) \widehat{\phi}(x)g_{\delta}(x).
%\end{equation}
%Note that for each $x \in \mathbb{R}$
%\begin{equation}\label{point2}
%lim_{\delta \rightarrow 0}(U^\delta \phi)(x)=(U \phi)(x)=\int_\mathbb{R} \frac{1}{|x-y|^{1-\gamma}}\phi(z)dz,
%\end{equation}
%and for $f\in \mathcal{F}$
%\begin{equation}\label{limit}
%(U f_\epsilon)(x)\overset{\epsilon \rightarrow 0}{\rightarrow} |x|^{\gamma-1},\quad x \in \mathbb{R}.
%\end{equation}
Now we can define the approximating functional which will be at the center of our investigation. Let $\epsilon,\delta \in (0,1),f\in \mathcal{F}$. Mimicking~\eqref{t1}, for any $\phi \in \mathcal{S}(\mathbb{R})$ and a pair of real-valued cadlag processes $\eta,\xi$ we put
\begin{equation}\label{QILTdefi}
\langle \Delta_{\epsilon,\delta}^f(\eta,\xi;T),\phi \rangle :=\int_0^T \int_0^T \phi(\eta_u)(V^\delta f_\epsilon)(\xi_v-\eta_u)dudv.
\end{equation}
For $\phi \geq 0$ the above integral converges pointwise as $\delta \rightarrow 0$ and then $\epsilon \rightarrow 0$ by Lemma~\ref{vprop} (to a possibly infinite limit) independently of the choice of $f$. This limit is given by
\begin{equation}\label{QILT}
\langle \Delta(\eta,\xi;T),\phi \rangle=\int_0^T \int_0^T\phi(\eta_u)|\xi_v-\eta_u|^{\gamma-1}dudv,\quad  \phi \in \mathcal{S}.
\end{equation}
Now we proceed to show that in the setting that interests us most ($\eta=x+\xi^1, \xi=y+\xi^2$ for independent symmetric $\alpha$-stable L\'{e}vy processes $\xi^1$ ,$\xi^2$ and $x,y\in \mathbb{R}$) the random variables given by~\eqref{QILTdefi} and~\eqref{QILT} are meaningful.
%and $\langle \Delta_{\epsilon,\delta}^f(\eta,\xi;T),\phi \rangle$ converges in $L^2(\Omega)$ as $\epsilon,\delta \rightarrow 0$ to a limit which is independent of the choice of $f$ then we denote this limit by $\langle \Delta(\eta,\xi;T),\phi \rangle$ and conclude (again using remark 3.2 from~\cite{PP}) that there exists an $\mathcal{S'}(\mathbb{R})$-valued process $\Delta(\eta,\xi)$ such that
%\begin{rem}\label{pointlim}
%By Lemma~\ref{vprop} we see that for positive $\phi$ in $C^\infty_c(\mathbb{R})$  the pointwise limit
%\begin{equation}
%\lim_{\epsilon,\delta\rightarrow 0} \langle \Delta_{\epsilon,\delta}^f(\eta,\xi;T),\phi \rangle
%\end{equation}
%exists and is finite if and only if $\int_0^T \int_0^T \phi(\eta_u)|\xi_v-\eta_u|^{\gamma-1}dudv$ is finite. Moreover this limit equals $\int_0^T \int_0^T \phi(\eta_u)|\xi_v-\eta_u|^{\gamma-1}dudv$.
%\end{rem} 
We can only show that $\Delta(\eta_1,\eta_2;T)$ exists as an $\mathcal{S'}$-valued random variable in the setting in which $\alpha,\beta \in (1/2,1)$. The convergece of $\langle \Delta_{\epsilon,\delta}^f(\eta_1,\eta_2;T),\phi \rangle$ in $L^2(\Omega)$ for fixed $x,y \in \mathbb{R}$ remains an open question. However, we will only need the following. 
\begin{lem}\label{L2conv}
Let $\xi^1,\xi^2$ be independent symmetric $\alpha$-stable processes with $\alpha \in (0,1)$.
\begin{itemize}
\item[(i)]
For every $\phi \in \mathcal{S}(\mathbb{R}),T>0$ the function given by
\begin{equation}
(x,y,\omega)\mapsto \langle \Delta_{\epsilon,\delta}^f(x+\xi^1,y+\xi^2;T),\phi \rangle
\end{equation}
converges in $L^2(\mathbb{R}^2 \times \Omega, \lambda_1 \otimes \lambda_1 \otimes \mathbb{P})$, as $\epsilon,\delta$ go to $0$, to the function 
\begin{equation}\label{fire}
(x,y,\omega) \mapsto \Delta(x,y,\omega)=\int_0^T \int_0^T \phi(x +\xi^1_r(\omega))|y +\xi^2_s(\omega) - x -\xi^2_r(\omega)|^{\frac{\beta-\alpha}{2}-1}drds.
\end{equation}
\item[(ii)]
The $L^2$-convergence as in (i) holds also if we replace $\phi$ by any function $\psi \in \mathcal{A}$. Moreover, for $\phi$ of the form $\phi=\psi_\kappa=\psi \ast g_\kappa$, where $\kappa \in (0,1)$ and $g \in \mathcal{F}$, the convergence is uniform in $\kappa$.
\end{itemize}
Here $\lambda_1$ is the one dimensional Lebesgue measure and $\mathbb{P}$ is the underlying probability measure.
\end{lem}

\begin{proof}[Proof of Lemma~\ref{L2conv}]
The proof is quite straightforward once we have established Lemma~\ref{dod}. We have
\begin{multline}\label{put}
\mathbb{E}\int_{\mathbb{R}^2}\left(\langle \Delta_{\epsilon,\delta}^f(x+\xi^1,y+\xi^2;T),\phi \rangle - \langle \Delta(x+\xi^1,y+\xi^2;T),\phi \rangle\right)^2 dxdy
\\
\leq \mathbb{E}\int_{\mathbb{R}^2}\Bigg(\int_{[0,T]^2}|\phi(x+\xi^1_u)|\Big|V_{f_\epsilon}^\delta(y+\xi^2_v-x-\xi^1_u)\\
-|y+\xi^2_v-x-\xi^1_u|^{\gamma-1}\Big|dudv\Bigg)^2dxdy\\
\leq  \mathbb{E}\int_{\mathbb{R}^2}\left(\int_{[0,T]^2}|\phi(x+\xi^1_u)|C_1(f,\gamma)|y+\xi^2_v-x-\xi^1_u|^{\gamma-1}dudv\right)^2dxdy,
\end{multline}
which is finite by Lemma~\ref{dod}. In the second inequality in~\eqref{put} we have used part (iii) of Lemma~\ref{vprop}. Now, using parts (i) and (ii) of the same Lemma and dominated convergence theorem we get the desired convergence. Evidently, the particular choice of $f$ from $\mathcal{F}$ is irrelevant. The rest of the proof is now straightforward.
\end{proof}

\subsection{Proof of Theorem~\ref{mainARP}}

\begin{proof}[Proof of Theorem~\ref{mainARP}]
The proof of Theorem~\ref{mainARP} more or less follows the line of reasoning of the proof of Theorem 3.5 in~\cite{FILT}. We will study the behavior of the following functional, which approximates the functional in the statement of Theorem~\ref{mainARP}:
\begin{equation}\label{mainapprox}
\eta_{f,\phi,\epsilon,\delta}^T:=\frac{1}{T}\sum_{j \neq k}\sigma_j \sigma_k \langle \Delta_{\epsilon,\delta}^f(x^j + \xi^j,x^k+\xi^k;T),\phi \rangle,
\end{equation}
where, as before $\epsilon,\delta \in (0,1), f \in \mathcal{F}, \phi \in \mathcal{S}(\mathbb{R}), T>0$. It is well defined by Lemma~\ref{khermiteperm}. Moreover, the above functional converges in $L^2(\Omega)$, as $\epsilon,\delta \rightarrow 0$ (uniformly in $T\geq 1$ and independently of the choice of $f \in \mathcal{F}$) to a random variable $\eta^T_{\phi}$ given by
\begin{equation}
\eta^T_{\phi}:=\frac{1}{T}\sum_{j \neq k}\sigma_j \sigma_k \langle \Delta(x^j + \xi^j,x^k+\xi^k;T),\phi \rangle,
\end{equation}
which again follows from Lemmas~\ref{L2conv} and~\ref{khermiteperm}. Recall that $\gamma=(\beta-\alpha)/2$. Let $\psi$ be any function from $\mathcal{A}$ and $\Psi_{\epsilon,\delta,\phi}^f(x,y):=\phi(x)(V^\delta f_\epsilon)(y-x)$, where $V^\delta$ is defined by~\eqref{vdef}. We will show the convergence, as $T\rightarrow \infty$, of the characteristic function of $\eta_{T,\psi_\kappa}$ to the characteristic function of the finite-dimensional distributions of the non-symmetric Rosenblatt process. Using inequality $|\mathbb{E}(e^{iX})-\mathbb{E}(e^{i\widetilde{X}})|\leq 2\mathbb{E}|X-\widetilde{X}|\leq 2\left(\mathbb{E}|X-\widetilde{X}|^2\right)^{\frac{1}{2}}$ (for real valued random variables $X$ and $\widetilde{X}$, it is enough to show that:

\begin{equation}\label{part0}
\underset{\kappa \rightarrow 0}{lim}\quad\underset{T \geq 1}{sup}\quad \mathbb{E}\left|\eta_\psi^T-\eta_{\psi_\kappa}^T\right|^2=0,
\end{equation}

\begin{equation}\label{part1}
\underset{\epsilon,\delta \rightarrow 0}{lim}\quad\underset{T \geq 1}{sup}\underset{\kappa \in (0,1)}{sup}\mathbb{E}\left|\eta_{f,\psi_\kappa,\epsilon,\delta}^T-\eta_{\psi_\kappa}^T\right|^2=0,
\end{equation}

\begin{equation}\label{part2}
\underset{T \rightarrow \infty}{lim}\underset{\kappa \in (0,1)}{sup}\mathbb{E}\left|\langle :X_T\otimes X_T:,\Psi_{\epsilon,\delta,\phi}^f \rangle - \eta_{f,\phi,\epsilon,\delta}^T\right|^2 = 0,\quad \epsilon,\delta>0,\phi \in \mathcal{S},
\end{equation}

\begin{equation}\label{part3}
\langle :X_T\otimes X_T:,\Psi_{\epsilon,\delta,\phi}^f \rangle \underset{T \rightarrow \infty}{\Rightarrow} \langle :X\otimes X:,\Psi_{\epsilon,\delta,\phi}^f \rangle,\quad \epsilon,\delta>0,\phi \in \mathcal{S},
\end{equation}

\begin{equation}\label{part4}
\underset{\epsilon,\delta,\kappa \rightarrow \infty}{lim}\mathbb{E}\left|\langle:X\otimes X:,\Psi_{\epsilon,\delta,\psi_\kappa}^f \rangle - \int_{\mathbb{R}^2}^{''} \widehat{\psi}(x+y)|y|^{-\gamma}Z_G(dx)Z_G(dy)\right|^2=0,
\end{equation}
where $Z_G$ is the random spectral measure as in Remark~\ref{spectralm}. 

Similarly as in the proof of Lemma~\ref{L2conv} it is easy to show (using dominated convergence theorem) that~\eqref{part0} holds. It is enough to notice that $|\psi_\kappa| \leq |\psi|$. Lemma~\ref{L2conv} and Lemma 8.1 from~\cite{PP} give us~\eqref{part1} (for details see the proof of equation (6.26) in~\cite{FILT}). The proof of~\eqref{part2} is very similar to the proof of equation (6.27) in~\cite{FILT} so we will only sketch it. Recalling~\eqref{x_t} we may write

\begin{eqnarray}
\langle :X_T\otimes X_T:,\Psi_{\epsilon,\delta,\phi}^f \rangle &=& \frac{1}{T} \int_{[0,T]^2}\Big(\sum_{j,k} \sigma_j \sigma_k \Psi_{\epsilon,\delta,\phi}^f(x^j+\xi^j_s,x^k+\xi^k_u) \nonumber \\
&& -\: \int_{\mathbb{R}} \mathbb{E}\Psi_{\epsilon,\delta,\phi}^f(x+\xi^1_s,x+\xi^1_u)dx \Big) dsdu \nonumber \\
&=& \eta_{T,\phi,\epsilon,\delta}^f + \frac{1}{T}\int_{[0,T]^2} \Big(\sum_j \Psi_{\epsilon,\delta,\phi}^f(x^j+\xi^j_s,x^j+\xi^j_u) \nonumber \\
&& -\: \int_\mathbb{R}\mathbb{E}\Psi_{\epsilon,\delta,\phi}^f(x+\xi^1_s,x+\xi^1_u)dx \Big) dsdu.
\end{eqnarray}
This implies that, again by Lemma~\ref{fractional} in the Appendix,
\begin{multline}
\mathbb{E}\left|\langle :X_T\otimes X_T:,\Psi_{\epsilon,\delta,\phi}^f \rangle - \eta_{T,\phi,\epsilon,\delta}^f\right|^2 = \\
\frac{1}{T} \int_{[0,T]^4} \int_\mathbb{R} \mathbb{E}\Psi_{\epsilon,\delta,\phi}^f(x+\xi^1_s,x+\xi^1_u)\Psi_{\epsilon,\delta,\phi}^f(x+\xi^1_r,x+\xi^1_v)dx dsdudrdv.
\end{multline}
The rest of the argument is exactly the same as in the proof of equation (6.26) in~\cite{FILT} because $V^\delta f_\epsilon \in \mathcal{S}(\mathbb{R})$ implies that $\Psi_{\epsilon,\delta,\phi}^f$ is in $\mathcal{S}(\mathbb{R}^2)$. The convergence in~\eqref{part3} follows from Lemma 6.3 in~\cite{FILT}. For the proof of~\eqref{part4} we will look at Hermite processes from the point of view of multiple Wiener-Ito integrals. From Theorem 4.7 in~\cite{MAJOR} it follows that for any $\Phi \in \mathcal{S}(\mathbb{R}^2)$
\begin{equation}
\langle :X\otimes X:,\Phi \rangle = \int_{\mathbb{R}^2}^{''} \widehat{\Phi}(x,y)Z_G(dx)Z_G(dy),
\end{equation}
where $Z_G$ is the random spectral measure, as in Remark~\ref{spectralm}, corresponding to the spectral measure $G(dx)=|x|^{-\alpha}dx$. Hence, using~\eqref{vtrans},
\begin{equation}
\langle :X\otimes X:,\Psi_{\epsilon,\delta,\psi_\kappa}^f \rangle =\frac{1}{2\pi} \int_{\mathbb{R}^2}^{''} \widehat{\psi_\kappa}(x+y)\widehat{f_\epsilon}(y)\widehat{h_\delta}(y)Z_G(dx)Z_G(dy).
\end{equation}
By dominated convergence and $L^2$ isometry $\langle :X\otimes X:,\Psi_{\epsilon,\delta,\psi_\kappa}^f \rangle$ converges in $L^2(\Omega)$ as $\epsilon,\delta,\kappa \rightarrow 0$ to $\frac{1}{2\pi}\int_{\mathbb{R}^2}^{''} \widehat{\psi}(x+y)|y|^{-\gamma}Z_G(dx)Z_G(dy)$, which by the change of variables formula for multiple Wiener-It\^o integrals (Theorem 4.4 in~\cite{MAJOR}), is equal to
\begin{equation}
\frac{1}{2\pi}\int_{\mathbb{R}^2}^{''} \widehat{\psi}(x+y)|x|^{-\frac{\alpha}{2}}|y|^{-\frac{\beta}{2}}\widehat{W}(dx)\widehat{W}(dy),
\end{equation}
where $\widehat{W}$ is a complex-valued Fourier transform of white noise (see discussion after equation (2.6) in~\cite{FILT}). When we replace $\psi$ with $\mathbf{1}_{[0,t]}$, then (using definition 3.8 in~\cite{BT} and spectral representation discussed in the following remarks) we can define
\begin{equation}
Z_t^T:=\eta_{T,\mathbf{1}_{[0,t]}},\quad t \geq 0,
\end{equation}
then $(Z_t^T)_{t \geq 0}$ converges (up to a constant), as $T \rightarrow \infty$, in the sense of finite dimensional distributions, to the non-symmetric $(\alpha,\beta)$-Rosenblatt process.
\end{proof}

%\subsection{ILT}

%In~\cite{FILT} it was shown that, up to a constant, we can get the Rosenblatt process with the parameter $H=\alpha \in (\frac{1}{2},1)$ as a limit (in the sense of finite-dimensional distributions) of a process given by 

%\begin{equation}
%\xi^T_t:=\frac{1}{T}\sum_{j \neq k}\sigma_j \sigma_k \langle \Lambda(x^j + \xi^j,x^k+\xi^k;T),\mathbf{1}_{[0,t]} \rangle,
%\end{equation}

%as $T \rightarrow \infty$. Here $\Lambda$ is an $\mathcal{S'}$-valued process called the \emph{intersection local time} of the symmetric $\alpha$-stable Levy process (see~\cite{FILT}). This was done for the symmetric Rosenblatt process,i.e. , the process given by the generalized Hermite kernel of the form $g(x_1,x_2)=|x_1|^{\alpha/2-1}|x_2|^{\alpha/2-1}\mathbf{1}_{\{x_1>0,x_2>0\}}$. We would like to generalize this result to the assymetric Rosenblatt process which has a kernel of the form

%\begin{equation}
%g(x_1,x_2)=|x_1|^{\alpha/2-1}|x_2|^{\beta/2-1}\mathbf{1}_{\{x_1>0,x_2>0\}},
%\end{equation}

%where $\alpha,\beta \in (0,1)$ and $\alpha + \beta>1$. Without losing generality we may assume that $\beta>\alpha$. We also put $\gamma=\frac{\beta-\alpha}{2}$. We will prove that up to a constant the functional given by

%\begin{equation}
%\xi_{T,\mathbf{1}_{(0,t)}} =\frac{1}{T}\sum_{j \neq k}\sigma_j \sigma_k \int_0^T \int_0^T \mathbf{1}_{(0,t)}(x^j+\xi^j_r)\frac{1}{|x^k+\xi^k_s-x^j-\xi^j_r|^{1-\gamma}}drds,
%\end{equation}

%converges in the sense of finite-dimensional distributions to the Assymetric $(\alpha,
%\beta)$-Rosenblatt process.

%\section{Notation}

\section{Representation of Hermite Processes}\label{RHP}

\subsection{$k$-intersection local time of independent $\alpha$-stable processes}\label{kherm}

Following \cite{PP} we would like to extend the definition of intersection local time to \emph{k-intersection local time}. For $\epsilon>0,\phi \in \mathcal{S}(\mathbb{R}), f \in \mathcal{F}$ and an integer $k \geq 2$ put
\begin{equation}\label{kiltdefi}
\Phi_{\epsilon,\phi}^f(x_1,\ldots,x_k):=\phi(x_1)f_\epsilon(x_2-x_1) \ldots f_\epsilon(x_k-x_1).
\end{equation}
Note that $\Phi_{\epsilon,\phi}^f \in \mathcal{S}(\mathbb{R}^k)$ and 
\begin{equation}
\widehat{\Phi_{\epsilon,\phi}^f}(x_1,\ldots,x_k)=\widehat{\phi}(x_1+\ldots+x_k)\widehat{f_\epsilon}(x_2) \ldots \widehat{f_\epsilon}(x_k).
\end{equation}
Using~\eqref{kiltdefi} we can define the \emph{approximate intersection local time} of $k$ real valued cadlag stochastic processes. For any cadlag processes $\rho_1,\ldots,\rho_k$ taking values in $\mathbb{R}$, $\phi \in \mathcal{S}(\mathbb{R})$ and $f \in \mathcal{F}$ we denote the approximate intersection local time at time $T>0$ by 
\begin{multline}
\left \langle \Lambda_\epsilon^f(\rho^1,\ldots,\rho^k;T),\phi \right \rangle =\\
=\int_{[0,T]^k}\phi(\rho^1_{s_1})f_\epsilon(\rho^2_{s_2}-\rho^1_{s_1})\ldots f_\epsilon(\rho^k_{s_k}-\rho^1_{s_1})ds_1\ldots ds_k.
\end{multline}
\begin{defi}
If there exists an $\mathcal{S'}$-valued random variable $\Lambda^{(k)}(\rho_1,\ldots,\rho_k)$ such that for each $\phi \in \mathcal{S}$ and $f \in \mathcal{F}$ $ \left \langle \Lambda_\epsilon^f(\rho^1,\ldots,\rho^k;T),\phi \right \rangle$ converges to \\
$\langle \Lambda^{(k)}(\rho_1,\ldots,\rho_k),\phi \rangle$ in $L^2(\Omega)$and the limit is independent of the choice of $f \in \mathcal{F}$ then $\Lambda^{(k)}$ is called the $k$-intersection local time of $\rho_1,\ldots,\rho_k$.
\end{defi}
We have the following extension of Proposition 5.1 in~\cite{PP}.
\begin{lem}\label{kiltex}
Let $\xi_1,\ldots,\xi_k$ be independent $\alpha$-stable L\'{e}vy processes with $\alpha \in (1-\frac{1}{k},1)$. Then for any starting points $x_1,\ldots,x_k \in \mathbb{R}$ and $\phi \in \mathcal{S}$ the k-intersection local time
 $\left \langle \Lambda^{(k)}(x_1+\xi^1,\ldots,x_k+\xi^k;T),\phi \right \rangle$ exists. Moreover $\Lambda^{(k)}$ can be evaluated for any function $\phi$ in $\mathcal{A}$.
\end{lem}
\begin{proof}
To prove the lemma it is enough to show that for any $f,g \in \mathcal{F}$, $x_1,\ldots,x_k \in \mathbb{R}$ and each $\phi \in \mathcal{S}$, the limit 
\begin{equation}\label{srw}
\lim_{\epsilon,\delta \rightarrow 0}\mathbb{E}\left(\left \langle \Lambda^f_\epsilon(x_1+\xi^1,\ldots,x_k+\xi^k;T),\phi \right \rangle \left \langle \Lambda^g_\delta(x_1+\xi^1,\ldots,x_k+\xi^k;T),\phi \right \rangle\right)
\end{equation}
exists and does not depend on the choice of $f$ and $g$. The proof is at first very similar to the proof of Proposition 3.3 and 5.1 in~\cite{PP}. Writing out the expectation in~\eqref{srw} using the $\alpha$-stable transition densities, passing to the Fourier transform, using Plancherel formula and then using the estimate
\begin{equation*}
\int_{[0,T]^2}|\widehat{\mu}_{s,u}(z,z')|dsdu \leq C(T)\frac{1}{1+|z+z'|^\alpha}(\frac{1}{1+|z|^\alpha}+\frac{1}{1+|z'|^\alpha}),
\end{equation*}
where $C(T)$ is a constant and $\mu_{s,u}$ is the law of $(\xi^1_s,\xi^1_u)$, the proof is reduced to showing that
\begin{eqnarray}\label{holder}
I&=&\int_{\mathbb{R}^{2k}} |\widehat{\phi}(z_1+\ldots+z_k)\widehat{\phi}(w_1+\ldots+w_k)|\\
&& \: \times b(z_1,w_1)\ldots b(z_k,w_k)dz_1\ldots dz_k dw_1 \ldots dw_k < \infty,\nonumber 
\end{eqnarray}
where $b(z,w)=C(T)\frac{1}{1+|z+w|^\alpha}(\frac{1}{1+|z|^\alpha}+\frac{1}{1+|w|^\alpha})$. To show that we will use H\"{o}lder inequality. First, fix some $\lambda \in (0,1)$. Rewrite equation~\eqref{holder} as follows.
\begin{eqnarray}\label{holder2}
I&=&\int_{\mathbb{R}^{2k}} |\widehat{\phi}(z_1+\ldots+z_k)^{\frac{k}{k}}\widehat{\phi}(w_1+\ldots+w_k)^{\frac{k}{k}}|\\
&& \: \times b(z_1,w_1)^{\lambda \frac{k}{k}}\ldots b(z_k,w_k)^{\lambda \frac{k}{k}}\nonumber \\
&& \: \times b(z_1,w_1)^{(1-\lambda) \frac{k-1}{k-1}}\ldots b(z_k,w_k)^{(1-\lambda) \frac{k-1}{k-1}}dz_1\ldots dz_k dw_1 \ldots dw_k. \nonumber 
\end{eqnarray}
The integrand can be written as $g_1(\mathbf{z},\mathbf{w})\ldots g_k(\mathbf{z},\mathbf{w})$, where 
\begin{eqnarray}
g_j(\mathbf{z},\mathbf{w})&=&|\widehat{\phi}(z_1+\ldots+z_k)\widehat{\phi}(w_1+\ldots+w_k)|^{\frac{1}{k}}h(\mathbf{z},\mathbf{w})^{\frac{\lambda}{k}}\\
&& \: \times b(z_1,w_1)^{(1-\lambda)\frac{1}{k-1}}\ldots b(z_{j-1},w_{j-1})^{(1-\lambda)\frac{1}{k-1}}\nonumber \\
&& \: \times b(z_{j+1},w_{j+1})^{(1-\lambda)\frac{1}{k-1}}\ldots b(z_k,w_k)^{(1-\lambda)\frac{1}{k-1}},\nonumber 
\end{eqnarray}
and $h(\mathbf{z},\mathbf{w})=b(z_1,w_1)\ldots b(z_k,w_k)$. By H\"{o}lder inequality
\begin{equation}\label{prodholder}
I\leq\prod_{j=1}^k \left(\int_{\mathbb{R}^{2k}}g(\mathbf{z},\mathbf{w})^k d\mathbf{z} d\mathbf{w} \right)^\frac{1}{k}
\end{equation}
For $\phi \in \mathcal{S}$ or $\mathcal{A}$, $|\widehat{\phi}(x)|\leq \frac{C}{1+|x|},\quad x \in \mathbb{R}$, where $C$ is a constant. Now, taking $\lambda$ close enough to $0$, we have $\frac{(1-\lambda)k\alpha}{k-1}>1$. This implies that each factor in~\eqref{prodholder} is finite.
\end{proof}

In fact we will not need this ``pointwise'' sort of convergence and we will only utilize a weaker result (which is an analogue of Lemma~\ref{L2conv}) to be able to formulate the main theorem of this section in a rigorous way. 
\begin{lem}\label{majorconv}
Assume that $\alpha \in (1-\frac{1}{k},1)$. Then $\left \langle \Lambda_\epsilon^f(\cdot+\xi^1,\ldots,\cdot+\xi^k;T),\phi \right \rangle$ converges in $L^2(\mathbb{R}^k \times \Omega,\lambda_k\otimes \mathbb{P})$ as $\epsilon \rightarrow 0$ for any $\phi \in \mathcal{S}(\mathbb{R})$ and the limit is independent of he choice of $f\in \mathcal{F}$. Moreover, if we replace $\phi$ by any function of the form $\psi_\kappa$ as in~\eqref{fdd}, the convergence is uniform in $\kappa \in (0,1)$. We also denote this limit by $\left \langle \Lambda^{(k)}(\cdot+\xi^1,\ldots,\cdot+\xi^k;T),\phi \right \rangle$. 
\end{lem}
The proof of this lemma is similar to the case when $k=2$ and amounts to showing that that
\begin{equation}\label{holdem2}
\int_{\mathbb{R}^k}|\widehat{\phi}(x_1+\ldots +x_k)|^2|x_1|^{-\alpha}\ldots |x_k|^{-\alpha}dx_1\ldots dx_k <\infty,
\end{equation}
for $\phi \in \mathcal{S}(\mathbb{R})$ or $\mathcal{A}$. In the symmetric case above~\eqref{holdem2} follows by H\"{o}lder inequality, similarly as in Lemma~\ref{kiltex}. Indeed, putting
\begin{equation*}
s_j(x_1,\ldots,x_k):=\prod_{i=1,i\neq j}^k \Big({\frac{1}{1+|x_i|^\alpha}}\Big)^{1/(k-1)}
%\Big({\frac{1}{1+|x_1|^\alpha}}\Big)^{1/(k-1)}\ldots\Big({\frac{1}{1+|x_{j-1}|^\alpha}}\Big)^{1/(k-1)}\Big({\frac{1}{1+|x_{j+1}|^\alpha}}\Big)^{1/(k-1)}\ldots\Big({\frac{1}{1+|x_k|^\alpha}}\Big)^{1/(k-1)},
\end{equation*}
$j=1,\ldots,k$, we get  
\begin{multline}
\int_{\mathbb{R}^k}|\widehat{\phi}(x_1+\ldots +x_k)|^2\frac{1}{1+|x_1|^{\alpha}}\ldots \frac{1}{1+|x_k|^{\alpha}}dx_1\ldots dx_k \leq \\
\leq \prod_{j=1}^k\left(\int_{\mathbb{R}^k}|\widehat{\phi}(x_1+\ldots +x_k)|^2 s_j(x_1,\ldots,x_k)^k dx_1\ldots dx_k\right)^\frac{1}{k},
\end{multline}
which is finite since, by assumption, $\alpha>1-\frac{1}{k}$.
\begin{rem}
In fact one can prove that for any choice of $\alpha_1,\ldots,\alpha_k \in (0,1)$ satisfying $\alpha_1+\ldots+\alpha_k>k-1$ and any $f\in L^1(\mathbb{R})$
\begin{equation}
\int_{\mathbb{R}^k}|\widehat{f}(x_1+\ldots +x_k)|^2|x_1|^{-\alpha_1}\ldots |x_k|^{-\alpha_k}dx_1 \ldots dx_k <\infty,
\end{equation}
but the proof is a little more complicated.
\end{rem}

In order to use the properties of $\mathcal{S'}$-valued random variables we introduce the following approximating functional:
\begin{equation}\label{kiltapprox}
\rho_{f,\epsilon,\phi}^T:=\frac{1}{T^{k/2}}\sum_{j_1 \neq  \ldots \neq j_k}\sigma_{j_1}\ldots \sigma_{j_k}\langle \Lambda_\epsilon^f(x^{j_1}+\xi^{j_1},\ldots,x^{j_k}+\xi^{j_k};T),\phi \rangle,
\end{equation}
where $T>0,\epsilon>0,f\in \mathcal{F},\phi \in \mathcal{S}(\mathbb{R})$. It is well defined by Lemmas~\ref{khermiteperm} and~\ref{majorconv}. The same Lemmas also show that the functional given by
\begin{equation}
\rho_{\psi_\kappa}^T:=\frac{1}{T^{k/2}}\sum_{j_1 \neq  \ldots \neq j_k}\sigma_{j_1}\ldots \sigma_{j_k}\langle \Lambda^{(k)}(x^{j_1}+\xi^{j_1},\ldots,x^{j_k}+\xi^{j_k};T),\psi_\kappa \rangle,
\end{equation}
is well defined for any $\kappa \in (0,1)$, $T>0$, $\psi \in \mathcal{A}$ and is an $L^2(\Omega)$-limit of the functional in~\ref{kiltapprox} with $\phi$ replaced by $\psi_\kappa$.
\subsection{Proof of the representation}

% DOWOD CZESC DRUGA
\begin{proof}[Proof of Theorem~\ref{mainK}]
The proof follows the footsteps of the proof of Theorem 3.5 in~\cite{FILT} with some necessary generalizations. From now on we fix $\alpha \in (1-\frac{1}{k},1)$ and $f \in \mathcal{F}$. We are going to prove the following claims:

\begin{equation}\label{pa1}
\underset{\kappa \rightarrow 0}{lim}\quad\underset{T \geq 1}{sup}\quad \mathbb{E}\left|\rho_{\psi}^T-\rho_{\psi_\kappa}^T\right|^2=0,
\end{equation}

\begin{equation}\label{pa2}
\underset{\epsilon \rightarrow 0}{lim}\quad\underset{T \geq 1}{sup}\quad \underset{\kappa \in (0,1)}{sup}\quad \mathbb{E}\left|\rho_{\psi_\kappa}^T-\rho_{f,\epsilon,\psi_\kappa}^T\right|^2=0,
\end{equation}

\begin{equation}\label{hardpart}
\lim_{T\rightarrow \infty} \underset{\kappa \in (0,1)}{sup}\mathbb{E}\left|\langle :X_T \otimes \ldots \otimes X_T:,\Phi_{\epsilon,\psi_\kappa}^f \rangle - \rho_{f,\epsilon,\psi_\kappa}^T\right|^2=0,\quad \epsilon >0,
\end{equation}

\begin{equation}\label{pa4}
\langle :X_T \otimes \ldots \otimes X_T:,\Phi_{\epsilon,\psi_\kappa}^f \rangle \Rightarrow \langle :X \otimes \ldots \otimes X:,\Phi_{\epsilon,\psi_\kappa}^f \rangle,\quad,\epsilon>0,\kappa >0,
\end{equation}

\begin{eqnarray}\label{pa5}
\lim_{\epsilon \rightarrow 0} \underset{\kappa \in (0,1)}{sup}  \mathbb{E}\bigg|\int_{\mathbb{R}^k}^{''}\widehat{\psi_\kappa}(x_1+\ldots +x_k)Z_G(dx_1)\ldots Z_G(dx_k)\nonumber\\
-\langle:X \otimes \ldots \otimes X:,\Phi_{\epsilon,\psi_\kappa}^f \rangle\bigg|^2=0,
\end{eqnarray}

\begin{eqnarray}\label{pa6}
\lim_{\kappa \rightarrow 0}\mathbb{E}\bigg|\int_{\mathbb{R}^k}^{''}\widehat{\psi_\kappa}(x_1+\ldots +x_k)Z_G(dx_1)\ldots Z_G(dx_k)-\nonumber \\
\int_{\mathbb{R}^k}^{''}\widehat{\psi}(x_1+\ldots +x_k)Z_G(dx_1)\ldots Z_G(dx_k)\bigg|^2=0.
\end{eqnarray}
Here $:Z\otimes \ldots \otimes Z:$ stands for the $k$-th Wick product of the random variable $Z$. For definition see equation~\eqref{defwick}. From Lemma~\ref{khermiteperm} we have (with $F$ replaced by $\langle (\Lambda_\epsilon^f-\Lambda),\psi_\kappa \rangle$) the following inequality:
\begin{eqnarray}
\mathbb{E}|\rho_{\psi_\kappa}^T-\rho_{f,\epsilon,\psi_\kappa}^T|^2 &\leq& \frac{2k!}{T^k}\int_{\mathbb{R}^k}\mathbb{E}|\langle (\Lambda_\epsilon^f-\Lambda^{(k)})(x_1+\xi^1,\ldots,x_k+\xi^k;T),\psi_\kappa \rangle|^2\nonumber \\
&\leq& 2k!2^k\int_{\mathbb{R}^k}|\widehat{\psi_\kappa}(x_1+\ldots +x_k)|^2|\widehat{f_\epsilon}(x_1+\ldots +x_k)-1|^2\nonumber\\
&& \times\:|x_1|^{-\alpha}\ldots |x_k|^{-\alpha}dx_1\ldots dx_k.
\end{eqnarray}
By~\eqref{holdem2}, dominated convergence theorem and the fact that $\widehat{\psi_\kappa}(z)\leq \widehat{\psi}(z)$ for $z\in\mathbb{R}$ we get~\eqref{pa2}. The proof of~\eqref{pa1} is very similar and we skip it.

The hardest part is to prove~\eqref{hardpart}. From~\cite{FILT} we know that (iii) holds for $k=2$. Let $\Phi$ be of the form
\begin{equation}\label{appform}
\Phi=\sum_{j=1}^m \phi^{(1,j)}\otimes \ldots \otimes \phi^{(k,j)},
\end{equation}
where each $\phi^{(s,t)}$ is in $\mathcal{S}(\mathbb{R})$ for $s=1,\ldots,k,\quad t=1,\ldots,m$. By definition
\begin{multline}\label{defwick}
\langle :X_T\otimes \ldots \otimes X_T:,\Phi \rangle = \\
\sum_{j=1}^m \sum_{A \in \mathcal{M}}(-1)^{|A|} \prod_{\{s,t\} \in A}\mathbb{E}(\langle X_T,\phi^{(s,j)} \rangle) \mathbb{E}(\langle X_T,\phi^{(t,j)} \rangle) 
\prod_{n\notin \cup A}\langle X_T,\phi^{(n,j)} \rangle,
\end{multline}
where $\mathcal{M}$ is the set of unordered pairs $\{s,t\} \subset \{1,\ldots,k\}$, such that all the elements in these pairs are distinct. In particular $|\bigcup A|=2|A|$. The sum above is over all distinct sets $A$ of this form including the empty set. If we define the approximating functional by
\begin{equation}
\rho_{\Phi}^T:=\frac{1}{T^{k/2}}\sum_{j_1 \neq \ldots \neq j_k}\sigma_{j_1}\ldots \sigma_{j_k}\int_{[0,T]^k}\Phi(x^{j_1}+\xi^{j_1}_{s_1},\ldots,x^{j_k}+\xi^{j_k}_{s_k})ds_1 \ldots ds_k,
\end{equation}
then one can easily see that
\begin{equation}\label{reduceform}
\mathbb{E}(\langle :X_T\otimes \ldots \otimes X_T:,\Phi \rangle -\rho_{\Phi}^T)^2=\mathbb{E}\langle :X_T\otimes \ldots \otimes X_T:,\Phi \rangle^2-\mathbb{E}(\rho_{\Phi}^T)^2.
\end{equation}
This follows from the fact that in $\rho_{\Phi}^T$ we have summation over distinct indices $\{j_1,\ldots,j_k\} \in \mathbb{N}$ and so the only nonzero terms in $\mathbb{E}(\langle :X_T\otimes \ldots \otimes X_T:,\Phi \rangle \rho_{\Phi}^T$ are those that correspond to $A=\emptyset$ in~\eqref{defwick}. Furthermore, if we recall the sum in~\eqref{x_t} defining $\langle X_T,\phi \rangle$ for $\phi \in \mathcal{S}(\mathbb{R})$, then it is obvious that $\mathbb{E}(\langle :X_T\otimes \ldots \otimes X_T:,\Phi \rangle \rho_{\Phi}^T=\mathbb{E}(\rho_{\Phi}^T)^2$. Let us denote $\mathbb{E}\langle :X_T\otimes \ldots \otimes X_T:,\Phi \rangle^2$ by $I$. Then

%Now, we would must notice a couple of things. First, computing $\mathbb{E}\langle :X_T\otimes \ldots \otimes X_T:,\Phi \rangle^2$ amounts to computing the sums over the \emph{diagonals} in the following sense.

\begin{eqnarray}\label{mainsquare}
I&=& \sum_{j=1}^m \sum_{j'=1}^m \sum_{A,A'\in \mathcal{M}}(-1)^{|A|}(-1)^{|A'|} \prod_{(s,t)\in A}\mathbb{E}\langle X_T,\phi^{(s,j)} \rangle \langle X_T,\phi^{(t,j)} \rangle \nonumber \\
&& \times \: \prod_{(s',t')\in A'}\mathbb{E}\langle X_T,\phi^{(s',j')} \rangle \langle X_T,\phi^{(t',j')} \rangle \nonumber \\
&& \times \: \mathbb{E}\left(\prod_{n \notin \cup A} \langle X_T,\phi^{(n,j)} \rangle \prod_{n' \notin \cup A'} \langle X_T,\phi^{(n',j')} \rangle\right).
\end{eqnarray}
Computing the last expected value in~\eqref{mainsquare} amounts to summation over different choices of the \emph{diagonals} just as in the proof of Lemma~\ref{khermiteperm}. To illustrate it consider first the case $A=\emptyset=A'$. Then, we have no covariances and are left with
\begin{eqnarray}\label{diagonals}
I_{\emptyset}&=&\mathbb{E}\Big(\sum_{j=1}^m \sum_{j'=1}^m \langle X_T, \phi^{(1,j)} \rangle \ldots \langle X_T, \phi^{(k,j)} \rangle \nonumber \\
&& \times\: \langle X_T, \phi^{(1,j)} \rangle \ldots \langle X_T, \phi^{(k,j)} \rangle \Big) \\
&=& \mathbb{E}\Big(\sum_{\substack{j_1,\ldots,j_k \\ j_{k+1},\ldots,j_{2k}}} \sigma_{j_1}\ldots \sigma_{j_k} \sigma_{j_{k+1}}\ldots \sigma_{j_{2k}}F_T(x^{j_1}+\xi^{j_1},\ldots,x^{j_k}+\xi^{j_k})\nonumber\\
&& \times\:F_T(x^{j_{k+1}}+\xi^{j_{k+1}},\ldots,x^{j_{2k}}+\xi^{j_{2k}}),\Big),\nonumber
\end{eqnarray}
where
\begin{multline}\label{ft}
F_T(x_1+\xi^1,\ldots,x_k+\xi^k):=\\
\frac{1}{T^{k/2}}\int_{[0,T]^{k}}\Phi(x_1+\xi^1_{s_1},\ldots,x_k+\xi^k_{s_k})ds_1 \ldots ds_k,\quad x_1,\ldots,x_k \in \mathbb{R}.
\end{multline}
%By~\eqref{defwick}
%\begin{equation}\label{decomp}
%\mathbb{E}\langle :X_T\otimes \ldots \otimes X_T:,\Phi \rangle^2=I_1+I_2+I_3+R,
%\end{equation}
%where
%\begin{eqnarray}
%I_1+I_2&=&\mathbb{E}(\sum_{j=1}^m \sum_{j'=1}^m \langle X_T, \phi^{(1,j)} \rangle \ldots \langle X_T, \phi^{(k,j)} \rangle \\
%&& \times\: \langle X_T, \phi^{(1,j)} \rangle \ldots \langle X_T, \phi^{(k,j)} \rangle) \\
%&=& \mathbb{E}(\sum_{\substack{j_1,\ldots,j_k \\ i_1,\ldots,i_k}} \sigma_{j_1}\ldots \sigma_{j_k} \sigma_{i_1}\ldots \sigma_{i_k}\int_{[0,T]^{2k}}F(x^{j_1}+\xi^{j_1},\ldots,x^{j_k}+\xi^{j_k})\nonumber\\
%&& \times\:F(x^{i_1}+\xi^{i_1},\ldots,x^{i_k}+\xi^{i_k})ds_1 \ldots ds_k dr_1\ldots dr_k,\label{difsum}
%\end{eqnarray}
%and
%\begin{multline}
%F_T(x_1+\xi^1,\ldots,x_k+\xi^k):=\\
%\frac{1}{T^{k/2}}\int_{[0,T]^{k}}\Phi(x_1+\xi^1_{s_1},\ldots,x_k+\xi^k_{s_k})ds_1 \ldots ds_k,\quad x_1,\ldots,x_k \in \mathbb{R}.
%\end{multline}
The only terms in the sum in~\eqref{diagonals}, whose expected values are non-zero, are those for which for every $l \in \{j_1,\ldots,j_k,j_{k+1},\ldots,j_{2k}\}$ there is an \emph{even} number of indices taking that value. This sum can be split into a finite number of sums over different \emph{diagonals}. To be precise, by a \emph{diagonal} $\mathcal{C}$ we mean a partition of $\{1,2,\ldots,2k\}$ into a disjoint family of subsets $C_1,\ldots,C_m$ of $\{1,2,\ldots,2k\}$ such that $|C_l|$ is even for $l=1,\ldots,m$. Then the term in~\eqref{diagonals} corresponding to this diagonal is given by
\begin{multline}\label{monster}
\sum_{\substack{j_{v_1^1}=\ldots=j_{v_{k_1}^1},\quad v_1^1,\ldots,v_{k_1}^1 \in C_1 \\ \ldots \\ j_{v_1^m}=\ldots=j_{v_{k_m}^m},\quad v_1^m,\ldots,v_{k_m}^m \in C_m }}\sigma_{j_1}\ldots \sigma_{j_k} \sigma_{j_{k+1}}\ldots \sigma_{j_{2k}} \\
\times F_T(x^{j_1}+\xi^{j_1},\ldots,x^{j_k}+\xi^{j_k})F_T(x^{j_{k+1}}+\xi^{j_{k+1}},\ldots,x^{j_{2k}}+\xi^{j_{2k}}).
\end{multline}
Now, a diagonal $\mathcal{C}$ is \emph{large} if any of the sets $C_1,\ldots,C_m$ has more than two elements. In due course we will see that the sums over these diagonals behave as $\frac{1}{T}$ as $T\rightarrow \infty$. All other diagonals are \emph{pairings} of different charges. This means that for these diagonals all $C_l$'s have exactly two elements. We will say that a pairing is \emph{normal} if for every $C \in \mathcal{C}$, $C$ has exactly one element from $\{1,\ldots,k\}$ and one element from $\{k+1,\ldots,2k\}$. All other non-large diagonals will be called \emph{non-normal pairings}. Notice that the choice of $A,A'$ in~\eqref{mainsquare} corresponds to fixing some particular part of the diagonal over which summation is being done. Looking at~\eqref{mainsquare} from this perspective, there will only be normal pairings in the sums corresponding to $A,A'=\emptyset$. We can use the same notation for sums that will emerge from the term $\mathbb{E}\left(\prod_{n \notin \cup A} \langle X_T,\phi^{(n,j)} \rangle \prod_{n' \notin \cup A'} \langle X_T,\phi^{(n',j')} \rangle\right)$ in~\eqref{mainsquare}. Putting all this together we can write
\begin{equation}\label{sort}
\mathbb{E}\langle :X_T\otimes \ldots \otimes X_T:,\Phi \rangle^2=I_1+I_2+R,
\end{equation}
where in $I_1$ we have only the sums over the diagonals which correspond to normal pairings. We see immediately that $I_1=\mathbb{E}\eta_{T,\Phi}^2$). $I_2$ corresponds to the sums over non-large non-normal pairings (notice that all the sums in~\eqref{mainsquare} with $A \neq \emptyset$ or $A' \neq \emptyset$ will be in $I_2$) and $R$ contains only sums over large diagonals. Notice that the terms in~\eqref{sort} may be written with the help of the function $F_T$ given by~\eqref{ft} and extend continuously to general $\Phi \in \mathcal{S}(\mathbb{R}^2)$, not necessarily of the form~\eqref{appform}.

We are going to show that $I_2=0$. Fix a non-large non-normal pairing $\mathcal{C}$. Assume that in the fixed diagonal over which the summation is being performed there are $n$ non-normal pairs $B$ formed between members of the sequence $(j_1,\ldots,j_k)$ and $n$ non-normal pairs $B'$ formed between members of the sequence $(j_{k+1},\ldots,i_{2k})$. All the other pairs (there are $k-2n$ of them) are normal. The sum over our fixed diagonal is going to appear in terms from~\eqref{mainsquare}  with $|A|,|A'|\leq n$. Fix $c,d \leq n$ and consider the summands in $I_2$ for which $|A|=c,|A'|=d$. The sum over $\mathcal{C}$ is going to appear in exactly $\binom{n}{c}\binom{n}{d}$ summands with $c+d=n$ with the sign equal to $(-1)^{c+d}$. This can be justified as follows. The sum over $\mathcal{C}$ can only appear in the terms with $A\subset B$ and $A'\subset B'$ and there will be $\binom{n}{c}$ choices of $A$ with $A\subset B$ and $\binom{n}{d}$ choices of $A'$ with $A'\subset B'$. We see that for each $0 \leq m \leq 2n$ with $m=|A|+|A'|$ the sum corresponding to our fixed diagonal will appear exactly $\sum_{l=0}^m \binom{n}{m-l}\binom{n}{l}=\binom{2n}{m}$ with sign equal to $(-1)^m$. Hence, the number of times (with signs taken into account) the sum over our fixed diagonal will appear in $I_2$ is exactly $\sum_{m=0}^{2n}(-1)^m\binom{2n}{m} =0$. This proves that $I_2=0$.

$R$ can be split into a finite number of sums over large diagonals, all of which have the property that the summation is taken over indices $(j_1,$$\ldots,$$j_k,$$j_{k+1},$$\ldots,$$j_{2k})$ such that at least four of them are equal. To finish the proof of~\eqref{hardpart} we fix $\epsilon \in (0,1)$ and take $\Phi=\Phi_{\epsilon,\psi_\kappa}^f$ (see~\eqref{kiltdefi}). It remains to show that $R=R_{T,\Phi_{\epsilon,\psi_\kappa}^f}$ converges to $0$ as $T \rightarrow \infty$ uniformly in $\kappa \in (0,1)$. Put $\theta(x):=\frac{1}{1+|x|^2},x \in \mathbb{R}$. Notice that
\begin{equation} 
|\Phi(x_1,\ldots,x_k)|\leq p(\Phi) \frac{1}{1+|x_1|^2}\ldots \frac{1}{1+|x_k|^2}=p(\Phi)\theta(x_1)\ldots \theta(x_k),
\end{equation}
where $p(\Phi)$ is a continuous seminorm on $\mathcal{S}(\mathbb{R}^k)$, given by
\begin{equation}
p(\Phi)=\sup_{x_1\in \mathbb{R},\ldots,x_k \in \mathbb{R}}\left|(1+|x_1|^2)\ldots(1+|x_k|^2)\Phi(x_1,\ldots,x_k)\right|.
\end{equation}
Thanks to this
\begin{equation}\label{sup}
\sup_{\kappa \in (0,1)}p(\Phi_{\epsilon,\psi_\kappa}^f)\leq C(\epsilon,f),
\end{equation}
with $C(\epsilon,f)$ being a constant depending only on $f$ and $\epsilon$ and independent of $\kappa$.
%\begin{equation}
%\Phi_{\psi_\kappa,\epsilon}^f(x_1,\ldots,x_k)\leq C(\epsilon,f)\frac{1}{1+|x_1|^2}\ldots \frac{1}{1+|x_k|^2},
%\end{equation}
To fix our attention, let us consider an example of a large diagonal with $k=3$. This diagonal is given by requiring that $j_1=j_2=j_5=j_6$ and $j_3=j_4$. Then the expected value of the sum corresponding to this diagonal is given by
\begin{equation}
\frac{1}{T^{\frac{3}{2}}}\int_{\mathbb{R}^2}\mathbb{E}\left(F_T(x_1+\xi^1,x_1+\xi^1,x_2+\xi^2)F_T(x_2+\xi^2,x_1+\xi^1,x_1+\xi^1)\right)dx_1dx_2.
\end{equation}
The absolute value of the above integral is no bigger than
\begin{multline}\label{split}
\frac{1}{T^{\frac{3}{2}}}p(\Phi)^2\int_{\mathbb{R}^2}\int_{[0,T]^6}\mathbb{E}\Big(\theta(x_1+\xi^1_{u_1})\theta(x_1+\xi^1_{u_2})\theta(x_2+\xi^2_{u_3}) \\
\times\theta(x_2+\xi^2_{u_4}) \theta(x_1+\xi^1_{u_5})\theta(x_1+\xi^1_{u_6})\Big)du_1\ldots du_6 dx_1dx_2.
\end{multline}
By~\eqref{sup}, for $\Phi=\Phi_{\epsilon,\psi_\kappa}^f$, the integral in~\eqref{split} can be bounded uniformly in $\kappa$ by an integral which (by independence) can be written as a product of two integrals times a constant $C(\epsilon,f)$. One of the factors of this product (the one correspoding to the pairing $j_3=j_4$) is bounded by a constant. The other is given by
\begin{multline*}
\frac{1}{T^2}\int_{[0,T]^4}\int_\mathbb{R}\mathbb{E}\left(\theta(x+\xi^1_s) \theta(x+\xi^1_u)\theta(x+\xi^1_r)\theta(x+\xi^1_v)\right)dxdsdudrdv.
\end{multline*}
Following the proof of Theorem 3.5 in~\cite{FILT} we see that the above is bounded by
\begin{equation}
\frac{1}{T^2}\int_0^T \int_\mathbb{R}\theta(x)U(\theta U(\theta U\theta))(x)dxds\leq \frac{1}{T}C_2,
\end{equation}
where $C_2$ is another constant, and $U$ is the potential of an $\alpha$-stable semigroup. The second inequality above follows from the fact that $U\psi$ is bounded. In the case of blocks larger than four the argument is very similar. To conclude, $R \leq \frac{C}{T}p(\Phi)^2$, where $C$ is a constant, which depends only on $k,\epsilon$ and $f$ (the bound $\frac{1}{T}$ was given in~\cite{FILT} only for the diagonal with the largest element consisting of four equal indexes, but having larger diagonals is even better which can easily be inferred from the proof of equation (6.27) in~\cite{FILT}). This means that we can write $\mathbb{E}\left(\langle :X_T\otimes \ldots \otimes X_T:,\Phi \rangle - \rho_{T,\Phi}\right)^2 \leq \frac{C}{T}p(\Phi)^2$ for $\Phi \in \mathcal{S}(\mathbb{R}^k)$.

To prove~\eqref{pa4} we will need the following generalization of lemma 6.3 in~\cite{FILT}
\begin{lem}
Let $(X_T)_{T \geq 1}$ be a family of $\mathcal{S'}$-valued random variables such that 
\begin{equation*}
\sup_{T \geq 1} \mathbb{E}\langle X_T,\phi \rangle^2 \leq p^2(\phi),\quad \phi \in \mathcal{S}(\mathbb{R}),
\end{equation*}
for some continuous Hilbertian seminorm $p$ on $\mathcal{S}(\mathbb{R})$. Suppose that $X_T \Rightarrow X$ and $\mathbb{E}\langle X_T,\phi \rangle^2 \rightarrow \mathbb{E}\langle X,\phi \rangle^2$ for $\phi \in \mathcal{S}(\mathbb{R})$ as $T \rightarrow \infty$. Then $X_T \otimes \ldots \otimes X_T$ and $X \otimes \ldots \otimes X$ are well defined and $:X_T \otimes \ldots \otimes X_T: \Rightarrow :X \otimes \ldots \otimes X:$ as $T \rightarrow \infty$.
\end{lem}
This together with Theorem~\ref{FBR} implies~\eqref{pa4}. We proceed to prove~\eqref{pa5} and~\eqref{pa6}. Fix $k \in \mathbb{N},k \geq 2$. Let $\alpha \in (1-\frac{1}{k},1)$ and let $(X_\phi)_{\phi \in \mathcal{S}(\mathbb{R})}$ be a generalized centered Gaussian random field over the Schwartz space with spectral measure $G(dx)=|x|^{-\alpha}dx$. Notice that, as before, using Theorem 4.7 in~\cite{MAJOR}, we might write
\begin{equation}
\langle :X\otimes \ldots \otimes X:,\Phi \rangle = \int_{\mathbb{R}^k}^{''}\widehat{\Phi}(x_1,\ldots,x_k)Z_G(dx_1)\ldots Z_G(dx_k),
\end{equation}
for any $\Phi \in \mathcal{S}(\mathbb{R}^k)$. Whenever $\int_{\mathbb{R}^k}^{''}|\widehat{\phi}(x_1+\ldots +x_k)|^2 G(dx_1)\ldots G(dx_k) < \infty$, we see that by dominated convergence theorem, 
\begin{equation}\label{spectralconv}
\langle :X\otimes \ldots \otimes X:,\Phi_{\epsilon,\phi}^f \rangle \overset{L^2(\Omega)}{\rightarrow} \int_{\mathbb{R}^k}^{''}\widehat{\phi}(x_1+\ldots +x_k)Z_G(dx_1)\ldots Z_G(dx_k).
\end{equation}
Given the above~\eqref{pa5} and~\eqref{pa6} follow immediately. Establishing~\eqref{pa1} - ~\eqref{pa6} shows that the finite-dimensional distributions of $(\eta_t^T)_{t \geq 0}$ converge to the finite dimensional distributions of the $k$-Hermite process given by~\eqref{spec_rep}
\end{proof}
\begin{rem}
In fact we have shown that for $\Phi \in \mathcal{S}(\mathbb{R}^k)$ the functional defined by
\begin{equation}
\rho_{T,\Phi}:=\frac{1}{T^{k/2}}\sum_{j_1 \neq \ldots \neq j_k}\sigma_{j_1}\ldots \sigma_{j_k}\int_{[0,T]^k}\Phi(x^{j_1}+\xi^{j_1}_{s_1},\ldots,x^{j_k}+\xi^{j_k}_{s_k})ds_1 \ldots ds_k,
\end{equation}
converges (up to a constant) in distribution to
%we have the following estimate
%\begin{equation}
%\mathbb{E}(\eta_{T,\Phi}^2)\leq \left(\frac{1}{2\pi}\right)^k k!\int_{\mathbb{R}^k}|\widehat{\widetilde{\Phi}}(z_1,\ldots,z_k)|^2|z_1|^{-\alpha}\ldots |z_k|^{-\alpha}dz_1 \ldots dz_k,
%\end{equation}
%for any $T>0$, where $\widetilde{\Phi}$ is the symmetrization of $\Phi$. Moreover, 
%\begin{equation}
%\lim_{T\rightarrow \infty} \mathbb{E}(\eta_{T,\Phi}^2)=\left(\frac{1}{2\pi}\right)^k k!\int_{\mathbb{R}^k}|\widehat{\widetilde{\Phi}}(z_1,\ldots,z_k)|^2|z_1|^{-\alpha}\ldots |z_k|^{-\alpha}dz_1 \ldots dz_k.
%\end{equation}
%\end{lem}
%\begin{lem}
%Let $\Phi \in \mathcal{S}(\mathbb{R}^k)$ and take $\eta_{T,\Phi}$ as in Lemma~\ref{etaapp}. Then
%\begin{equation}
\begin{equation}
\int_{\mathbb{R}^k}^{''}\widehat{\Phi}(z_1,\ldots,z_k)Z_G(dz_1)\ldots Z_G(dz_k),
\end{equation}
as $T\rightarrow \infty$. 
\end{rem}

\newpage

\appendix
\section{Appendix}
We would like to have the generalizations of some of the facts used in~\cite{PP} and~\cite{FILT}. First we state the generalized version of Lemma 8.1 from~\cite{PP}, which is the simplified version of the so called Mecke-Palm formula (see for instance equation (2.10) in~\cite{PL}). 
\begin{lem}\label{fractional}
Let $(x^j)$ be a Poisson system with intensity $\mu$ on $\mathbb{R}^d, d \geq 1$. If $F$ is in $L^1(\mathbb{R}^{kd},\mu^{\otimes k})$, then $\mathbb{E}\left(\sum_{j_1 \neq j_2 \neq \ldots \neq j_k}|F(x^{j_1},\ldots,x^{j_k})|\right)<\infty$ and
\begin{equation*}
\mathbb{E}\left(\sum_{j_1 \neq j_2 \neq \ldots \neq j_k}G(x^{j_1},\ldots,x^{j_k})\right)=\int_{\mathbb{R}^{kd}}F(x_1,\ldots,x_k)\mu(dx_1)\ldots \mu(dx_k).
\end{equation*}
\end{lem}

Assume that $\xi^1,\ldots,\xi^k$ are independent symmetric $\alpha$-stbale L\'{e}vy processes with $\alpha \in (1-\frac{1}{k},1)$ and $k\geq 2$ is an integer. Moreover, let $(\sigma_j)_{j \in \mathbb{N}}$ be a sequence of i.i.d. random variables independent of $\xi^1,\ldots,\xi^k$ and such that $\mathbb{P}(\sigma_1=1)=\mathbb{P}(\sigma_1=-1)$. We then have the following.

\begin{lem}\label{khermiteperm}
For any $F(\cdot + \xi^1,\ldots,\cdot + \xi^k) \in L^2(\mathbb{R}^k \times \Omega,\lambda_k\otimes \mathbb{P})$ the series 
\begin{equation*}
\sum_{j_1 \neq  \ldots \neq j_k}\sigma_{j_1}\ldots \sigma_{j_k}F(x^{j_1}+\xi^{j_1},\ldots,x^{j_k}+\xi^{j_k})
\end{equation*}
converges in $L^2(\Omega)$, and
\begin{multline*}
\mathbb{E}\left(\sum_{j_1 \neq  \ldots \neq j_k}\sigma_{j_1}\ldots \sigma_{j_k}F(x^{j_1}+\xi^{j_1},\ldots,x^{j_k}+\xi^{j_k})\right)^2= \\
=\int_{\mathbb{R}^k}\mathbb{E}\left(\sum_{\pi \in \Pi(k)}F(x_1+\xi^1,\ldots,x_k+\xi^k)\times F(x_{\pi_1}+\xi^{\pi_1},\ldots,x_{\pi_k}+\xi^{\pi_k})\right)\\
dx_1\ldots dx_k,
\end{multline*}
where the summation in the second integral is over all permutations $\pi$ of \\
$\lbrace 1,2,\ldots,k \rbrace$.
\end{lem}
This lemma follows immediately from Lemma~\ref{fractional} and the fact that the $\sigma_j$s are independent of $\xi^1,\ldots,\xi^k$.
%\section{The relationship between Wick polynomials and multiple Wiener-Ito integrals}

%\begin{lem}
%For $\Phi=\phi^{(1)}\otimes \ldots \otimes \phi^{(k)}$, where each $\phi^{(j)} \in \mathcal{S}(\mathbb{R})$ we have the following:

%\begin{equation}
%\langle X \otimes \ldots \otimes X,\Phi \rangle =\int_{\mathbb{R}^k}^{"}\widehat{\Phi}(z_1,\ldots,z_k)Z_G(dz_1)\ldots Z_G(dz_k),
%\end{equation}
%where $X$ is the generalized centered Gaussian random field with the spectral measure given by $G(dz)$. This equality can be generalized (by $L^2$ approximation) to any $\Phi \in \mathcal{S}(\mathbb{R}^k)$.
%\end{lem}

%The above lemma is a direct consequence of theorem 4.7 in [Major].

\end{document}